\documentclass{article}
\usepackage[utf8]{inputenc}
\usepackage[a4paper,margin=1in]{geometry}

\usepackage[affil-it]{authblk}

\usepackage[english]{babel}

\usepackage[onehalfspacing]{setspace}

\usepackage{amsmath,amssymb,amsthm,amsfonts,mathtools,mathrsfs}
\usepackage{graphics}
\usepackage{subcaption}
\usepackage{bm}
\usepackage{refcount}
\usepackage{booktabs}
\usepackage{enumitem}
\usepackage{adjustbox}
\usepackage{xcolor}
\usepackage{tikz}

\usepackage[short,nocomma]{optidef}

\usepackage[utf8]{inputenc} %input
\usepackage{amsthm}
\newtheorem{theorem}{Theorem}
\newtheorem*{conjecture*}{Conjecture}
\newtheorem{lemma}[theorem]{Lemma}
\newtheorem{corollary}[theorem]{Corollary}

\newtheorem{observation}[theorem]{Observation}
\theoremstyle{definition}
\newtheorem{definition}[theorem]{Definition}

\usepackage{hyperref}
\usepackage[authoryear,round]{natbib}
\usepackage{algorithm}
\usepackage[noend]{algpseudocode}

\title{A Compact Cycle Formulation for the Multiperiodic Event Scheduling Problem}
\author[1]{Rolf Nelson van Lieshout}
\author[2]{Niels Lindner}

\affil[1]{Department of Operations, Planning, Accounting, and Control,\protect\\School of Industrial Engineering, Eindhoven University of Technology \protect\\ r.n.v.lieshout@tue.nl\vspace{\baselineskip}}
\affil[2]{Zuse Institute Berlin \protect\\lindner@zib.de}

\begin{document}

\maketitle

\begin{abstract}
The Periodic Event Scheduling Problem (PESP) is a fundamental model in periodic timetabling for public transport systems, assuming a common period across all events. However, real-world networks often feature heterogeneous service frequencies. This paper studies the Multiperiodic Event Scheduling Problem (MPESP), a generalization of PESP that allows each event to recur at its own individual period. While more expressive, MPESP presents new modeling challenges due to the loss of a global period.
We present a cycle-based formulation for MPESP that extends the strongest known formulation for PESP and, in contrast to existing approaches, is valid for any MPESP instance.
Crucially, the formulation requires a cycle basis derived from a spanning tree satisfying specific structural properties, which we formalize and algorithmically construct, extending the concept of sharp spanning trees to rooted instances. 
We further prove a multiperiodic analogue of the cycle periodicity property.
Our new formulation solves nearly all tested instances—including several large-scale real-world public transport networks—to optimality or with small optimality gaps, dramatically outperforming existing arc-based models. The results demonstrate the practical potential of MPESP in capturing heterogeneous frequencies without resorting to artificial event duplication.
\end{abstract}

\section{Introduction}
Ever since its introduction by \cite{serafini1989mathematical}, the Periodic Event Scheduling Problem (PESP) has become a cornerstone of periodic timetabling models for public transport planning \citep{duarte2025fifty}.
 PESP is defined on a digraph, or \textit{event-activity network} $G=(V,A)$, bounds $l,u \in \mathbb{R}^A$, weights $w\in \mathbb{R}^A$ and a period time $T$, and seeks to find a \textit{periodic timetable} $\pi\in \mathbb{R}^V$ and a \textit{periodic tension} $x\in \mathbb{R}^A$ such that 
\begin{enumerate}
    \item[(1)] $\pi_j - \pi_i \equiv x_{ij} \bmod T$ for all $(i,j) \in A$,
    \item[(2)] $l\leq x\leq u$, 
    \item[(3)] $w^\top x$ is minimum,
\end{enumerate}
or decide that no such $(\pi,x)$ exists. %An event $i\in V$ scheduled at $\pi_i$, also occurs at $\pi_i+Tz$ for $z\in \mathbb{Z}.$

In the context of public transport timetabling, events represent departures or arrivals of services at stations, activities represent, e.g., driving, dwelling or transferring, bounds represent the minimum and maximum duration of those activities, and the period represents the cycle time of the timetable. If the weights additionally represent the (expected) number of passengers that use activities, the objective corresponds to minimizing passenger travel time. Figure~\ref{fig:fromLineToEAN} presents a public transport network with three bidirectional lines, and the associated event-activity network, including all possible transfer activities at the center station. Assuming $T=60$ minutes, if the departure of the orange line at the leftmost station is scheduled at time 19, this event recurs every hour at xx:19. 

\begin{figure}[!t]
    \centering
    % First subfigure
    \begin{subfigure}[t]{0.45\textwidth}
        \centering
        \includegraphics[width=\linewidth]{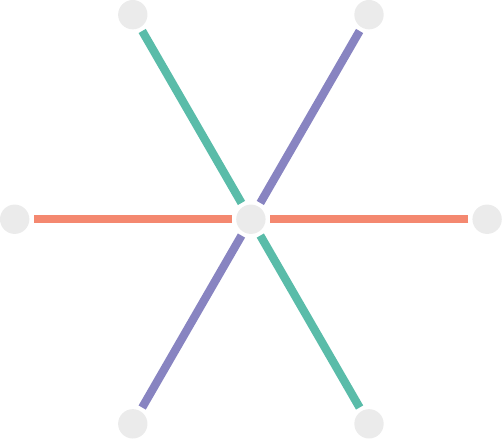}
        \caption{Three bidirectional lines, same frequency.}
        \label{fig:fromLineToEANa}
    \end{subfigure}
    \hfill
    % Second subfigure
    \begin{subfigure}[t]{0.45\textwidth}
        \centering
        \includegraphics[width=\linewidth]{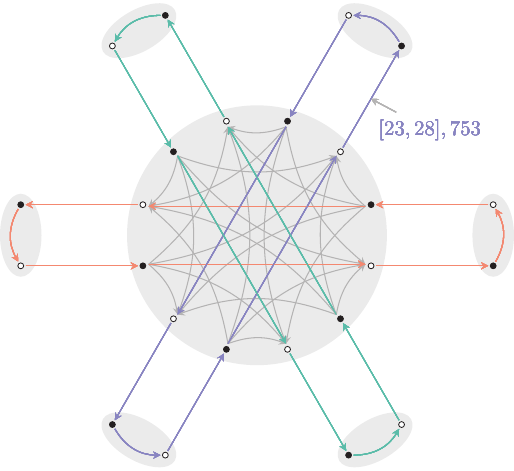}
        \caption{Event-activity network. The highlighted driving activity has 753 expected passengers and a duration between 23 and 28.}
        \label{fig:fromLineToEANb}
    \end{subfigure}
    \caption{Line network and the corresponding event-activity network}
    \label{fig:fromLineToEAN}
\end{figure}

This paper focuses on the \textit{Multiperiodic Event Scheduling Problem} (MPESP), a generalization of PESP in which events may occur at different individual frequencies. Instead of a global period $T$, each event $i \in V$ is associated with an event-specific period $T_i \in \mathbb{N}$, meaning the event recurs at times $\pi_i + z T_i$ for any integer $z$. Defining $T_a := \gcd\{T_i,T_j\}$ for all $(i,j)\in A$, MPESP replaces condition (1) in PESP by

\begin{itemize}[label={},leftmargin=*, align=left]
    \item[(MP1)]\label{item:mp} $\pi_j - \pi_i \equiv x_a \bmod T_a$ for all $a = (i,j) \in A$.
\end{itemize}
As in the single-period case, we call a vector $x \in \mathbb R^A$ a \emph{periodic tension} if there is a \emph{periodic timetable} $\pi \in \mathbb R^V$ satisfying (MP\ref{item:mp}).
MPESP provides a more compact and natural network representation for real-world public transport systems where services operate at different frequencies—for example, urban metro lines may run every 5 minutes, while regional trains arrive every 30 or 60 minutes. In the classical PESP, such heterogeneity is handled by duplicating higher-frequency events and linking them via synchronization activities, which significantly increases model size and complexity. Figure~\ref{fig:hetPesp} demonstrates this: a network with one hourly and one half-hourly service results in a much larger event-activity network using PESP compared to MPESP.

\begin{figure}[!ht]
    \centering
    % First subfigure
    \begin{subfigure}[t]{0.29\textwidth}
        \centering
        \includegraphics[width=0.8\linewidth]{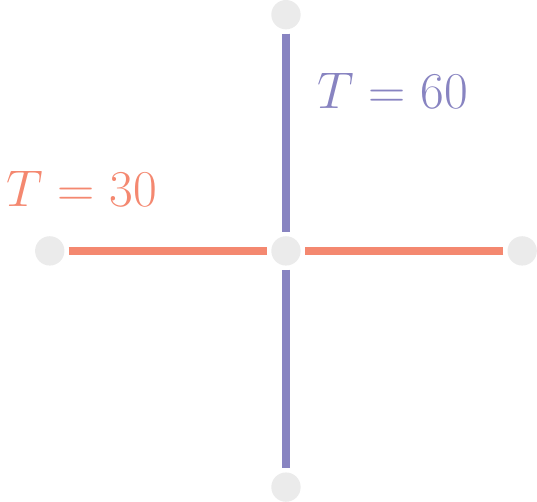}
        %\vspace{20pt}
        \caption{Two lines with different frequencies.}
        \label{fig:hetPespa}
    \end{subfigure}
    \hfill
    % Second subfigure
    \begin{subfigure}[t]{0.34\textwidth}
        \centering
        \includegraphics[width=\linewidth]{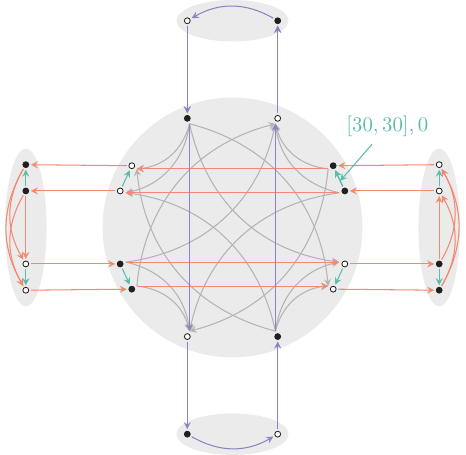}
        \caption{Event-activity network PESP.}
        \label{fig:hetPespb}
    \end{subfigure}
     \hfill
    % Third subfigure
    \begin{subfigure}[t]{0.34\textwidth}
        \centering
        \includegraphics[width=\linewidth]{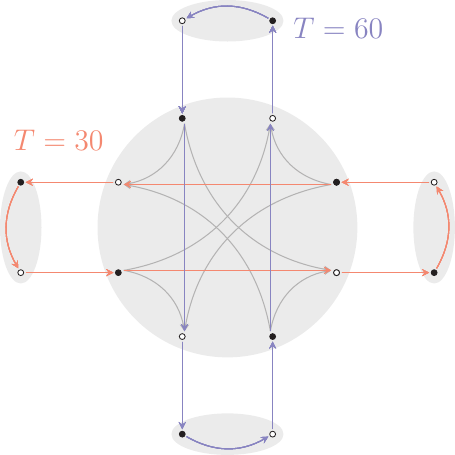}
        \caption{Event-activity network MPESP.}
        \label{fig:hetPespc}
    \end{subfigure}
    \caption{Line network with heterogeneous frequencies and the corresponding event-activity networks for PESP and MPESP.}
    \label{fig:hetPesp}
\end{figure}

A slight variation on MPESP, using individual arc frequencies instead of event frequencies, was introduced in the original PESP paper by \cite{serafini1989mathematical} as the Extended PESP (EPESP). \cite{nachtigall1996periodic} studied MPESP, although using the name EPESP, and proposed a branch-and-bound approach to minimize transfer times in an otherwise unconstrained setting. \cite{siebert2013experimental} used an arc-based formulation for EPESP, but did not report substantial computational gains. For MPESP, \cite{haenelt_taktfahrplanoptimierung_2004} investigated cutting planes and characterized those cycle bases that guarantee the cycle periodicity property by means of the Hermite normal form. On the algorithmic side, \cite{galli_strong_2010} introduced \emph{sharp} spanning trees to construct such cycle bases, gave a procedure to compute sharp trees when a certain divisibility condition on the event periods is satisfied, and achieved running time improvements. Other contributions on periodic timetabling with heterogeneous frequencies revert to the PESP framework with duplicated events and synchronization arcs, incurring a significant computational overhead, see, e.g., \cite{bortoletto_et_al,fuchs2022enhancing,kroon2014flexible,peeters2003cyclic,vanlieshout2021integrated}.

In this paper, we present a formulation for MPESP that significantly improves computational performance. It extends the strongest known formulation for standard PESP—the \textit{cycle formulation}—to the multiperiodic setting. 
To this end, we develop algorithms to construct \emph{sharp} spanning trees in both \textit{harmonic} and \textit{rooted} instances, and prove that any MPESP instance can be transformed into a rooted one by adding at most one auxiliary event and a limited number of arcs. Since periodic tensions are characterized by the cycle periodicity property for fundamental cycle bases of sharp spanning trees,
%As sharp spanning trees guarantee cycle periodicity for their associated fundamental cycle basis,
we establish the correctness of our approach, providing a cycle formulation for arbitrary MPESP instances.
We further show that in the multiperiodic setting, sharp spanning trees are exactly those spanning trees that allow for timetables being reconstructed from a tension by a tree traversal.
In general, for any MPESP instance, including those that do not admit a sharp spanning tree, we prove that the cycle periodicity property for all cycles naturally characterizes periodic tensions in the multiperiodic case, which has been an open question so far. The proof is constructive and allows to compute a timetable from a tension by solving systems of simultaneous congruences.

% While any integral cycle basis suffices for standard PESP, we show that a valid MPESP formulation requires a cycle basis formed from fundamental cycles of a spanning tree satisfying additional structural constraints, referred to as \textit{niceness}. We develop algorithms to construct nice spanning trees in both \textit{harmonic} and \textit{rooted} instances, and prove that any MPESP instance can be transformed into a rooted one by adding at most one auxiliary event and a limited number of arcs.

The proposed formulation is evaluated on two artificial and five real-world public transport networks, including the Swiss railway network, the Athens metro, and the full public transport system of Stuttgart. The cycle formulation solves all but the two largest instances to optimality within an average runtime of 15 seconds. In contrast, the arc formulation typically fails to close the gap within a one-hour time limit. On the largest instances, the cycle formulation yields smaller optimality gaps, while the arc formulation occasionally produces slightly better primal solutions. In the TimPassLib benchmark suite for integrated timetabling and passenger routing, a simple heuristic based on alternating timetabling and routing phases yields four new best-known solutions.

Beyond its strong computational performance, this paper also illustrates that MPESP naturally accommodates key features of real-world public transport systems, such as transfer and headway requirements, partial line synchronization, and passenger routing—without the need for duplicating high-frequency events. However, MPESP's compactness comes with a caveat: while it accurately captures passenger travel times in most cases, it may underestimate travel time for passengers making multiple transfers. Nevertheless, the numerical results indicate that the reduction in model size and complexity outweighs this limitation, making MPESP a practically effective tool for timetable design.

The remainder of this paper is structured as follows. Section~\ref{sec:back} describes background related to PESP. Section~\ref{sec:cform} recalls sharp spanning trees and presents the cycle formulation for MPESP. Section~\ref{sec:nice} develops an algorithm to find sharp spanning trees for any MPESP instance. Our multiperiodic version of the cycle periodicity property is proven in Section~\ref{sec:gcd}. Section~\ref{sec:ptt} highlights the modeling expressiveness of MPESP for public transport timetabling. Section~\ref{sec:cpu} presents the computational results. Finally, Section~\ref{sec:con} concludes the paper. %and mentions some open problems. 

\section{Background}
\label{sec:back}

This section summarizes important background for PESP. For more comprehensive treatments, we refer to, e.g., the monographs \cite{nachtigall_periodic_1998,peeters2003cyclic,liebchen_periodic_2006}, \cite{liebchen2008modeling} for modeling considerations, \cite{lindner_analysis_2022} for complexity results, \cite{lindner2025split} for polyhedral analysis, \cite{borndorfer_concurrent_2020} for state-of-the-art implementations, and \cite{liebchen2008first} for a real-world application.

Throughout this paper, we assume that all bounds are integer, which can be achieved by scaling if all bounds are rational. Moreover, we assume without loss of generality that the event-activity network $G$ is connected, and that $0\leq u_a-l_a \leq T-1$. The following mixed-integer program represents the \textit{arc formulation} for PESP: 
\begin{mini!}
	%
	% Variable
	{}
	%
	% Objec\mathbf{T}ive
	{ \sum_{a\in A} w_ax_a, \label{eq:objpesp}}
	%
	% Model label
	{\label{formulation:pesp}}
	%
	% Function closure for objective
	{}
	%
	% Constraints
	\addConstraint
	{~}
	{x_a = \pi_j-\pi_i+Tp_a\quad \label{eq:linkPESP}}
	{ \forall a=(i,j) \in A,}
	\addConstraint
	{l_a\leq ~}
	{ x_a\leq u_a \label{eq:boundspesp}}
	{\forall a\in A,}
    \addConstraint
	{~}
	{\pi_i \in \mathbb{R} \label{eq:pidomain}}
	{\forall i\in V,}
 \addConstraint
	{~}
	{p_a \in \mathbb{Z} \label{eq:pdomain}}
	{\forall a\in A.}
\end{mini!}
In this formulation, the integer variables $p_a$ take the role of the modulo operator. Symmetries can be broken by requiring $0\leq \pi_i \leq T-1$ for all $i\in V$ and setting $\pi_r=0$ for some arbitrary $r\in V$.

A fundamental structural property of PESP is \textit{cycle periodicity}. For any cycle $C$ in $G$, with forward arcs $C^+$ and backward arcs $C^-$, summing Constraints \eqref{eq:linkPESP} yields:
\begin{equation}
\label{eq:cp}
\sum_{a \in C^+} x_a - \sum_{a \in C^-} x_a \equiv 0 \bmod{T}.
\end{equation}
Conversely, if a tension vector $x$ satisfies cycle periodicity \eqref{eq:cp} for all cycles in $G$, then a feasible timetable $\pi$ always exists. Such a timetable can be constructed by traversing a spanning tree of $G$: choosing a root $r \in V$, setting $\pi_r = 0$, and computing $\pi_i$ for $i \in V \setminus \{r\}$ as the path distance modulo $T$ from $r$ in the tree, using  $x$ as arc weights.
%(modulo $T$ to ensure feasibility within $[0, T-1]$). %This implies that in the arc formulation, one can impose $p_a=0$ for all $a$ in some arbitrary spanning tree of $G$.

The \textit{cycle formulation} for PESP incorporates the condition for cycle periodicity as a constraint, and requires this constraint for a carefully constructed subset $\mathcal{B}$ of the set of all cycles:
\begin{mini!}
	%
	% Variable
	{}
	%
	% Objective
	{ \sum_{a\in A} w_ax_a, \label{eq:objPespCycle}}
	%
	% Model label
	{\label{formulation:tensionPespCycle}}
	%
	% Function closure for objective
	{}
	%
	% Constraints
	\addConstraint
	{\sum_{a\in C^+}x_a- \sum_{a\in C^-}x_a}
	{= Tz_C \quad \label{eq:cpfPesp}}
	{ \forall C \in \mathcal{B},}
	\addConstraint
	{l_a\leq x_a}
	{\leq u_a \label{eq:boundsPespCycle}}
	{\forall a\in A,}
        \addConstraint
	{z_C}
	{\in \mathbb{Z} \label{eq:boundsQPesp}}
	{\forall C \in \mathcal{B}.}
\end{mini!}
This formulation is valid if $\mathcal{B}$ is an \textit{integral cycle basis}, so that every cycle in $G$ is an integer linear combination of the basis cycles \citep{liebchen_integral_2009}. Formally, let $\gamma_C \in \{0,\pm1\}^A$ denote the vector that encodes the cycle $C$, where $\gamma_{C,a}=0$ if $a\notin C$ and $\gamma_{C,a}=\pm1$ if $a\in C$, the sign indicating whether $a$ is a forward or backward arc in $C$. Then, the set $\{\gamma_1,...,\gamma_k\}$ is an \textit{integral cycle basis} if and only if for all cycles $C$ in $G$ there exist integer coefficients $\lambda_j^C \in \mathbb Z$ such that $\gamma_C = \sum_{j=1}^k \lambda_j^C \gamma_j$.

An often used method for finding an integral cycle basis is to choose some spanning tree $H$ of $G$, and to let $\mathcal{B}$ be the set of cycles that are \textit{fundamental} with respect to $H$, i.e., all cycles that can be obtained by adding some co-tree arc $a\notin H$ to $H$. Such a cycle basis is called \textit{strictly fundamental} and is guaranteed to be integral \citep{kavitha_cycle_2009}.

The cycle formulation can be strengthened by adding Odijk's cycle cuts $a_C \leq z_C \leq b_C$, where 
\begin{equation}
    \label{ref:odijk}
    a_C=\left\lceil\frac{\sum_{a\in C^+}l_a- \sum_{a\in C^-}u_a}{T}\right\rceil \text{ and } b_C=\left\lfloor\frac{\sum_{a\in C^+}u_a- \sum_{a\in C^-}l_a}{T}\right\rfloor,
\end{equation}
which can also be seen to be (rank-1) Chvátal-Gomory cuts \citep{odijk1996constraint,lindner2025split}. The product $\prod_{C \in \mathcal{B}}\left(b_C-a_C\right)$, also known as the \textit{width} of $\mathcal{B}$, represents the number of values that $z$ can take, which motivates selecting $\mathcal{B}$ such that the width is small \citep{liebchen_integral_2009}. The complexity of finding a cycle basis that minimizes the width is still open. Therefore, a common heuristic is to use the cycle basis induced by a minimum spanning tree that is minimal with respect to arc weights $u_a-l_a$. 

Empirically, the cycle formulation shows better computational performance than the arc formulation, which is why it is the current method of choice for solving large-scale timetabling problems. 
\section{A Cycle Formulation for MPESP}
\label{sec:cform}
In this section, we develop the cycle formulation for MPESP. Recall that MPESP introduces a vector $\boldsymbol{T}\in \mathbb{N}^V$, which indicates the period time of each event. Define $T_a := \gcd\{T_i,T_j\}$ for all $a=(i,j) \in A$. A straightforward extension of Formulation~(\ref{formulation:pesp}) provides the arc formulation of MPESP: 
\begin{mini!}
	%
	% Variable
	{}
	%
	% Objec\mathbf{T}ive
	{ \sum_{a\in A} w_ax_a, \label{eq:objmpesp}}
	%
	% Model label
	{\label{formulation:MPpesp}}
	%
	% Function closure for objective
	{}
	%
	% Constraints
	\addConstraint
	{~}
	{x_a = \pi_j-\pi_i+T_ap_a\quad \label{eq:link}}
	{ \forall a=(i,j) \in A,}
	\addConstraint
	{l_a\leq ~}
	{ x_a\leq u_a \label{eq:bounds-mpesp}}
	{\forall a\in A,}
    \addConstraint
	{~}
	{\pi_i \in \mathbb{R} \label{eq:pidomain-mpesp}}
	{\forall i\in V,}
 \addConstraint
	{~}
	{p_a \in \mathbb{Z} \label{eq:pdomain-mpesp}}
	{\forall a\in A.}
\end{mini!}

Shifting the attention to cycles, define $T_C := \gcd\left\{T_a | a \in C\right\}$ for cycles $C$ in $G$. The following is the MPESP equivalent of the cycle periodicity property:
   %For an arbitrary tension vector $x$, I refer to a cycle that satisfies this condition as \textit{periodic}.

\begin{definition}
    A vector $x \in \mathbb R^A$ satisfies the $\boldsymbol{T}$-\emph{cycle periodicity property} if for all cycles $C$ in $G$ holds
     \begin{equation}
    \label{eq:mpesp-cycle-periodicity}
     \sum_{a\in C^+}x_a-\sum_{a\in C^-}x_a \equiv 0 \bmod T_C,
 \end{equation}
 or more concisely $\gamma_C^\top x \equiv 0 \bmod T_C$.
\end{definition}
 Summing Constraint~(\ref{eq:link}) for all arcs along a cycle, we obtain the following.
\begin{observation}
    \label{obs:cycle-periodicity}
    Any periodic tension $x$ for MPESP satisfies $\gamma_C^\top x \equiv  0 \bmod T_C$ for all cycles $C$ in $G$.
\end{observation}

We will show later in Section~\ref{sec:gcd} that the converse is true as well: If a vector $x$ satisfies the $\boldsymbol{T}$-cycle periodicity property, then $x$ is a tension, i.e., there is a timetable $\pi \in \mathbb R^V$ such that (MP\ref{item:mp}) holds.

However, to develop a useful cycle formulation, we need to single out cycle bases that imply the $\boldsymbol{T}$-cycle periodicity property. In contrary to PESP, in the MPESP it is no longer true that a tension $x$ that satisfies $\boldsymbol{T}$-cycle periodicity can be transformed into a feasible timetable $\pi$ by a graph traversal along an arbitrary spanning tree of $G$. However, \textit{certain} trees do produce feasible timetables. 

%TODO: Terminology tension/timetable vs. feasible tension/timetable

This is illustrated in Figure~\ref{fig:traversals}. The instance contains three cycles, all of which satisfying periodicity in the tension depicted in Figures~\ref{fig:traversalsa} and ~\ref{fig:traversalsb}, e.g., the tension along the cycle $1 \rightarrow 2 \rightarrow 3 \rightarrow 1$ is 10, an integer multiple of the greatest common divisor of all periods associated to the cycle. The timetable obtained using the tree in Figure~\ref{fig:traversalsa} is infeasible since $\pi_1-\pi_4 \not\equiv x_{41} \bmod T_{41}$. Conversely, the timetable produced by the tree in Figure~\ref{fig:traversalsb} is feasible.

\begin{figure}[h]
    \centering
     % First subfigure
    \begin{subfigure}[t]{0.42\textwidth}
        \centering
        \includegraphics[width=\linewidth]{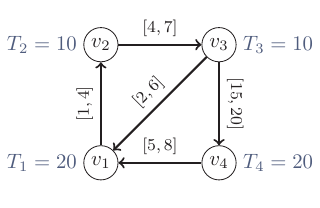}
        \caption{Instance.}
        \label{fig:traversalsInst}
    \end{subfigure}
    \hfill
    % Second subfigure
    \begin{subfigure}[t]{0.25\textwidth}
        \centering
        \includegraphics[width=\linewidth]{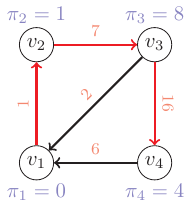}
        \caption{Infeasible, see arc (4,1).}
        \label{fig:traversalsa}
    \end{subfigure}
    \hfill
    % Third subfigure
    \begin{subfigure}[t]{0.26\textwidth}
        \centering
        \includegraphics[width=\linewidth]{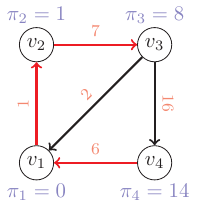}
        \caption{Feasible.}
        \label{fig:traversalsb}
    \end{subfigure}
    \caption{An instance (a) and periodic tension (b, c) for which the tree (red arcs) in (b) yields an infeasible timetable, while the tree in (c) yields a feasible one.}
    \label{fig:traversals}
\end{figure}

The following defines spanning trees whose traversal results in feasible timetables:

\begin{definition}[\citealp{galli_strong_2010}]
     A spanning tree $H$ of $G$ is 
     %$\boldsymbol{T}$-\textit{nice}
     \emph{sharp}
     if for every co-tree arc $a\notin H$, the cycle $C$ created by adding $a$ to $H$ satisfies $T_C = T_a$.
\end{definition}

Note that in the uniform-period case every spanning tree is sharp.
Before we discuss how to find sharp spanning trees in Section~\ref{sec:nice}, we recall the following:

\begin{theorem}[\citealp{galli_strong_2010}]
    \label{thm:galli-stiller}
    Suppose $\mathcal{B}$ is the fundamental cycle basis induced by a sharp spanning tree $H$. Let $x \in \mathbb R^A$. Then the following are equivalent:
    \begin{enumerate}[label={\normalfont(\roman*)}]
        \item The $\boldsymbol{T}$-cycle periodicity property \eqref{eq:mpesp-cycle-periodicity} holds for $x$.
        \item $\gamma_C^\top x \equiv 0 \bmod T_C$ for all $C\in \mathcal{B}$.
        \item The timetable $\pi$ obtained from $x$ by a graph traversal along $H$ is feasible.
        \item There is a timetable $\pi$ satisfying \textnormal{(MP\ref{item:mp})}.
    \end{enumerate}
\end{theorem}

A direct implication of the previous theorem is that the following cycle formulation for MPESP is valid, provided that $\mathcal{B}$ is a cycle basis induced by a sharp
%$\boldsymbol{T}$-nice 
spanning tree.
\begin{mini!}
	%
	% Variable
	{}
	%
	% Objective
	{ \sum_{a\in A} w_ax_a, \label{eq:obj}}
	%
	% Model label
	{\label{formulation:tension}}
	%
	% Function closure for objective
	{}
	%
	% Constraints
	\addConstraint
	{\sum_{a\in C^+}x_a- \sum_{a\in C^-}x_a}
	{= T_Cz_C \quad \label{eq:cpf}}
	{ \forall C \in \mathcal{B},}
	\addConstraint
	{l_a\leq x_a}
	{\leq u_a \label{eq:bounds}}
	{\forall a\in A,}
        \addConstraint
	{z_C}
	{\in \mathbb{Z} \label{eq:boundsQ}}
	{\forall C \in \mathcal{B}.}
\end{mini!}

% \begin{corollary}
%     \label{thm:nice-is-valid}
%     If $\mathcal{B}$ is a cycle basis induced by a sharp
%     %$\boldsymbol{T}$-nice
%     spanning tree, the cycle formulation \eqref{formulation:tension} is valid.
% \end{corollary}
% \begin{proof}
%     Consider any tension $x$ that satisfies (\ref{eq:cpf}), i.e. $\gamma_C^\top x \equiv 0 \mod T_C$ for all $C\in \mathcal{B}$. Then, 
%     %by Lemma~\ref{lem:someMeansAll}, $\gamma_C^\top x \equiv 0 \mod T_C$ for all cycles $C$ in $G$. Consequently, Lemma~\ref{lem:traversal}
%     Theorem~\ref{thm:galli-stiller}
%     applies and the timetable $\pi$ produced by a graph traversal along a 
%     %$\boldsymbol{T}$-nice 
%     sharp spanning tree, which $G$ must contain since $\mathcal{B}$ is induced by such a tree, is feasible.  
% \end{proof}

We remark that Odijk's cycle cuts naturally generalize to MPESP and read $a_C \leq z_C \leq b_C$, where 
\begin{equation}
    \label{ref:odijk-mpesp}
    a_C=\left\lceil\frac{\sum_{a\in C^+}l_a- \sum_{a\in C^-}u_a}{T_C}\right\rceil \text{ and } b_C=\left\lfloor\frac{\sum_{a\in C^+}u_a- \sum_{a\in C^-}l_a}{T_C}\right\rfloor.
\end{equation} 

We conclude this section by contributing a converse of Theorem~\ref{thm:galli-stiller}, showing that sharp spanning trees are exactly those that allow for a tree traversal to reconstruct a timetable:

\begin{theorem}
    Let $H$ be a spanning tree of $G$ with fundamental cycle basis $\mathcal B$. Suppose that for all tensions $x$ satisfying \eqref{formulation:tension} the timetable $\pi$ obtained by traversing $G$ along $H$ is feasible. Then $H$ is sharp.
\end{theorem}
\begin{proof}
Let $a^c = (i, j)$ be a co-tree arc of $H$ and choose a tree arc $a^t$ on the fundamental cycle $C$ of $a^c$. Set
\begin{equation}
     x_a \coloneqq \begin{cases}
        T_C & \text{ if } a = a^t, \\
        -T_C & \text{ if } a \notin H, a \neq a^c,\text{and } a^t \text{ is forward on the fundamental cycle of } a,\\
        T_C & \text{ if } a \notin H, a \neq a^c,\text{and } a^t \text{ is backward on the fundamental cycle of } a,\\
        0 & \text{ otherwise.}
    \end{cases}
\end{equation}
    Then $x$ satisfies \eqref{formulation:tension}, so that we obtain a timetable $\pi$ by graph traversal with 
    $\pm T_C = \pi_j - \pi_i \equiv 0 \bmod T_{a^c}$,
    hence $T_{a^c}$ divides $T_C$. Since $T_C$ always divides $T_{a^c}$, we conclude $T_{a^c} = T_C$.
\end{proof}
\section{Finding Sharp Trees}
\label{sec:nice}
As seen in Section~\ref{sec:cform}, we can use a cycle formulation for MPESP as long as we can find a sharp spanning tree.
However, Figure~\ref{fig:non-sharp} shows a simple MPESP instance that admits no sharp spanning tree: For the single cycle $C$, we have $T_C = 1$, but $T_a > 1$ for all arcs $a$.

\begin{figure}[b]
    \centering
   \includegraphics[width=0.5\textwidth]{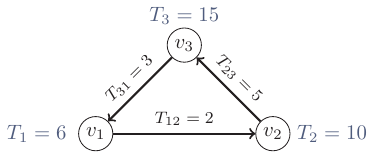}
    \caption{Instance without sharp spanning tree.}
    \label{fig:non-sharp}
\end{figure}

In this section, we  describe two classes of MPESP instances that admit sharp spanning trees: \textit{harmonic} instances and \textit{rooted} instances. In addition, we show how any instance can be transformed into a rooted instance. For defining harmonic instances, let  $\mathcal{T}:=\{T_i | i\in V\}$.

\begin{definition}
    The MPESP instance $(G,\boldsymbol{T},l,u,w)$ is \textit{harmonic} if $\mathcal{T}$ is totally ordered under divisibility; that is, for every $T,T'\in \mathcal{T}$, either $T$ divides $T'$ or $T'$ divides $T$.
\end{definition}

Harmonic MPESP instances are called \emph{nested} by \cite{galli_strong_2010}. An example of a harmonic instance is any instance with periods 15, 30 and 60. Such instances arise commonly in the context of public transport if lines have a frequency of once, twice or four times per hour. Here, it is straightforward to find sharp spanning trees: Sort the arcs according to $T_a$ in non-increasing order, and add arcs to the tree according to this order, skipping arcs that would create cycles. 
This results in a much simpler procedure than the one presented by \cite{galli_strong_2010}.

\begin{theorem}
    If the instance $(G,\boldsymbol{T},l,u,w)$ is harmonic, any maximum weight spanning tree algorithm with respect to arc weights $T_a$ finds a sharp spanning tree. 
\end{theorem}
\begin{proof}
    Let $H$ denote the spanning tree resulting from a maximum spanning tree algorithm, and suppose instead that $T_a>T_C$ for some co-tree arc $a \notin H$, and $C$ is the fundamental cycle of $a$. This implies that $T_a$ has higher weight than any of the arcs on the path between the endpoints of $a$ over $H$, violating the optimality of $H$. 
\end{proof}

For non-harmonic instances, it is required to impose additional restrictions on their structure to be able to always find sharp spanning trees. For $T\in \mathcal{T}$, let $G^T=\left(V^T,A^T\right)$ denote the subgraph induced by all nodes $i\in V$ such that $T_i = T$ and let $G_k^T$ denote its connected components for $k=1,...,m_T$.  Define $L:=\text{lcm}\left\{T_i| i\in V\right\}$. Moreover, it is useful to consider the (undirected) quotient graph $\mathcal{G}=(\mathcal{V},\mathcal{E})$, where each node $G_k^T\in \mathcal{V}$ represents a connected component of $G^T$, and the edge set $\mathcal{E}$ contains an edge between $G_k^T$ and $G_{k'}^{T'}$ if and only if there exists at least one arc $(i, j) \in A$ with one endpoint in $G_k^T$ and the other endpoint in $G_{k'}^{T'}$. Figure~\ref{fig:quotEx} presents an example of an instance and the associated quotient graph. 

\begin{figure}[ht]
    \centering
    % First row
    \begin{minipage}{0.53\textwidth}
        \centering
        \includegraphics[width=\linewidth]{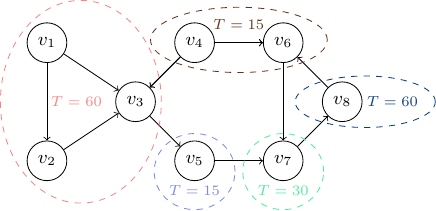}
        \subcaption{Instance.}
                    \label{fig:quotExa}

    \end{minipage}
    \begin{minipage}{0.45\textwidth}
        \centering
        \includegraphics[width=0.67\linewidth]{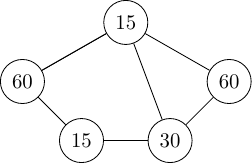}
        \subcaption{Quotient graph.}
                    \label{fig:quotExb}

    \end{minipage}
    \caption{Transformation from instance to quotient graph (periods indicated in nodes).}
    \label{fig:quotEx}
\end{figure}
\begin{definition}
    The MPESP instance $(G,\boldsymbol{T},l,u,w)$ is \textit{rooted} if  the following conditions hold: 
    \begin{itemize}
        \item[(i)] $L\in \mathcal{T}$,
        \item[(ii)] $G^L$ is connected,
        \item[(iii)] for all components $G_k^T\in \mathcal{V}$ such that $T<L$,  the quotient graph contains an edge $\left(G_k^T, G_{k'}^{Tq}\right)$ for some $k'$ and some $q$, where $q$ is an integer larger than 1.
    \end{itemize} 
\end{definition}
The third condition implies that every component with period $T<L$ is connected to some component with a period that is a multiple of $T$. Figure~\ref{fig:quotients} provides examples of quotient graphs that correspond to rooted and non-rooted instances. The instances corresponding to Figure~\ref{fig:quotientsc}, \ref{fig:quotientsd} and \ref{fig:quotientse} are non-rooted because conditions (i), (ii) and (iii) fail, respectively. 

\begin{figure}[ht]
    \centering
    % First row
    \begin{minipage}{0.45\textwidth}
        \centering
        \includegraphics[width=0.5\linewidth]{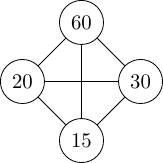}
        \subcaption{Rooted.}
                    \label{fig:quotientsa}

    \end{minipage}
    \begin{minipage}{0.45\textwidth}
        \centering
        \includegraphics[width=0.67\linewidth]{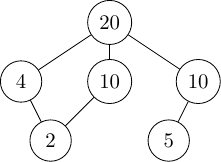}
        \subcaption{Rooted.}
                    \label{fig:quotientsb}

    \end{minipage}
    
    \bigskip
    
    % Second row
    \begin{minipage}{0.3\textwidth}
        \centering
        \includegraphics[width=0.75\linewidth]{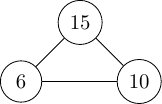}
        \subcaption{Non-rooted.}
            \label{fig:quotientsc}
    \end{minipage}
    \hfill
    \begin{minipage}{0.3\textwidth}
        \centering
        \includegraphics[width=0.6\linewidth]{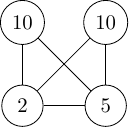}
        \subcaption{Non-rooted.}
            \label{fig:quotientsd}
    \end{minipage}
    \hfill
    \begin{minipage}{0.3\textwidth}
        \centering
        \includegraphics[width=\linewidth]{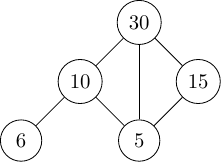}
        \subcaption{Non-rooted.}
            \label{fig:quotientse}
    \end{minipage}

    \caption{Quotient graphs corresponding to rooted and non-rooted instances.}
    \label{fig:quotients}
\end{figure}

The notion of the \textit{leader} of a component in the quotient graph is the final ingredient for finding sharp spanning trees. 
\begin{definition}
    The \textit{leader} of component $G_k^T\in \mathcal{V}$ where $T\neq L$ is the neighbor $G_{k'}^{Tq}$ in $\mathcal{G}$ that minimizes $q$ where $q$ is an integer greater than 1, breaking ties arbitrarily. 
\end{definition}
  %This is illustrated in Figure~\ref{fig:leader} 

\begin{figure}[ht]
    \centering
    % First row
    \begin{minipage}{0.45\textwidth}
        \centering
        \includegraphics[width=0.5\linewidth]{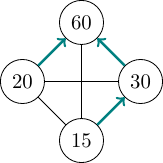}
        \subcaption{Corresponding to Figures~\ref{fig:quotientsa}.}
        \label{fig:leadera}
    \end{minipage}
    \begin{minipage}{0.45\textwidth}
        \centering
        \includegraphics[width=0.67\linewidth]{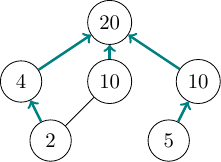}
        \subcaption{Corresponding to Figure~\ref{fig:quotientsb}.}
    \end{minipage}

    \caption{Leader relationships in the quotient graph.}
    \label{fig:leader}
\end{figure}

In other words, the leader of some component $G_{k}^T$ is a neighboring component whose period is the nearest multiple of $T$. Figure~\ref{fig:leader} illustrates the leader relationship. Clearly, the arcs that indicate the leaders form a tree rooted at the least common multiple $L$. This is true in general, since every component has exactly one leader, except $L$, which has no leader. Moreover, this tree resembles a sharp spanning tree: with slight abuse of notation, every co-tree edge $e\in \mathcal{E}$ satisfies $T_e=T_{C_e}$, where $C_e$ is the cycle generated by $e$. This is also true in general, since the path over the tree only visits multiples of one endpoint of the co-tree edge going up the tree, and only multiples of the other endpoint going down, thereby not affecting the greatest common divisor. It is also easy to see using Figures~\ref{fig:quotientsc}, \ref{fig:quotientsd} and \ref{fig:quotientse} why this would fail in non-rooted instances.

This motivates the following algorithm for constructing a sharp spanning tree in the original instance. First, find spanning trees in all components $G_k^T$ of $G$. Then, for every component that is not the least common multiple, add one arc that connects it with its leader. This is illustrated in Figure~\ref{fig:alg}, and formalized in the following theorem.

\begin{figure}[ht]
    \centering
    % First row
    \begin{minipage}{0.45\textwidth}
        \centering
        \includegraphics[width=0.5\linewidth]{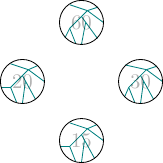}
        \subcaption{Find spanning trees in components.}
        \label{fig:alg1}
    \end{minipage}
    \begin{minipage}{0.45\textwidth}
        \centering
        \includegraphics[width=0.5\linewidth]{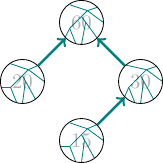}
        \subcaption{Connect with leaders.}
    \end{minipage}

    \caption{Finding a sharp spanning tree in a rooted instance.}
    \label{fig:alg}
\end{figure}

% \begin{theorem}
% \label{lem:nice}
%     If the instance $(G,\boldsymbol{T},l,u,w)$ is rooted, a spanning tree $H$ of $G$ is $\boldsymbol{T}$-nice if 
%     \begin{enumerate}[label={\normalfont(\roman*)}]
%     \item $H^T$ is a spanning forest of $G^T$ for all $T\in\mathcal{T}$, 
%     \item $H$ contains exactly one arc with one endpoint in $G_k^T$ and the other endpoint in $G_{k'}^{T'}$ if $G_{k'}^{T'}$ is the leader of $G_k^T$, or vice versa.
%     \end{enumerate}
% \end{theorem}
% \begin{proof}
% Consider a co-tree arc $a\notin H$ and consider the cycle $C_a$ generated by $a$. 
% In case $a\in A^T$ for some $T\in \mathcal{T}$, it follows from (i) that $C_a$ is contained entirely in $G^T$, obtaining that $T_{C_a} =  T = T_a$. In the other case, $a=(i,j)$ connects $G^{T_i}$ and $G^{T_j}$ for $T_i\neq T_j$. By (ii) the path in $H$ between $i$ and $j$ only traverses components whose period is a multiple of $T_i$ or $T_j$. Therefore, $T_{C_a} =  \gcd\{T_i,T_j\}  = T_a$. 
% \end{proof}
\begin{theorem}
    \label{thm:nice}
    If the instance $(G,\boldsymbol{T},l,u,w)$ is rooted, a spanning tree $H$ of $G$ is sharp if for each arc $(i, j) \in H$ that connects two different components $G_{k}^{T}$ and $G_{k'}^{T'}$ holds that $G_{k'}^{T'}$ is the leader of $G_k^T$, or vice versa.
    In particular, every rooted instance admits a sharp spanning tree.
\end{theorem}
\begin{proof}

    Consider a co-tree arc $a = (i, j) \notin H$ and the corresponding fundamental cycle $C$. The assumption implies that for each component $G_k^T$, there is at most one arc in $H$ that points to a component $G_{k'}^{T'}$ with $T' > T$, and this $T'$ is divisible by $T$. In particular, for all nodes $k$ on the unique $i$-$j$-path $P$ in $H$ it holds that $T_k$ is a multiple of $T_i$ or of $T_j$. We conclude that $T_C = \gcd\{T_k \mid k \in P\} = \gcd\{T_i, T_j\} = T_a$.
    Finally, note that in a rooted instance the earlier described procedure finds such a spanning tree, hence every rooted instance admits a sharp spanning tree. 
    \end{proof}

Crucially, any non-rooted instance can be transformed to a rooted instance by adding a node for the least common multiple, and/or adding arcs with non-restrictive bounds $[0,T_a-1]$ to ensure sufficient connectivity of the event-activity network, as illustrated in Figure~\ref{fig:trans}.

\begin{corollary}
    Any MPESP instance can be transformed to a rooted instance by adding at most one node and at most $\mathcal{O}(|V|)$ arcs $a$ with bounds $[0, T_a - 1]$.
\end{corollary}
\noindent Therefore, any instance of MPESP is solvable through a cycle formulation. 

Finally, note that to find cycle bases of small width, the spanning tree heuristic for PESP can easily be extended to MPESP. For harmonic instances, one can use an adaptation of Kruskal's algorithm that uses a lexicographic ordering of all arcs, first according to their period $T_a$ (descending), then according to their span $u_a-l_a$ (ascending). For rooted instances, the spanning forest $G^T$ for all $T\in \mathcal{T}$ can be found using any minimum spanning forest algorithm with respect to the spans $u_a-l_a$. The tree can then be completed by sorting the remaining arcs according to their spans and iteratively adding the first admissible arc. 

\begin{figure}[ht]
    \centering
    % First row
    \begin{minipage}{0.3\textwidth}
        \centering
        \includegraphics[width=0.75\linewidth]{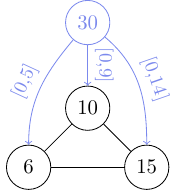}
        \subcaption{Corresponding to Figure~\ref{fig:quotientsc}.}
            \label{fig:transc}
    \end{minipage}
    \hfill
    \begin{minipage}{0.3\textwidth}
        \centering
        \includegraphics[width=0.55\linewidth]{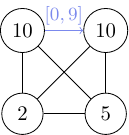}
        \subcaption{Corresponding to Figure~\ref{fig:quotientsd}.}
            \label{fig:transd}
    \end{minipage}
    \hfill
    \begin{minipage}{0.3\textwidth}
        \centering
        \includegraphics[width=\linewidth]{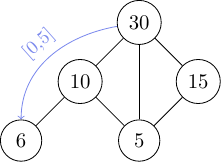}
        \subcaption{Corresponding to Figure~\ref{fig:quotientse}.}
            \label{fig:transe}
    \end{minipage}

    \caption{From non-rooted to rooted instances.}
    \label{fig:trans}
\end{figure}
\section{The Multiperiodic Cycle Periodicity Property}
\label{sec:gcd}

The aim of this section is to prove the following theorem, providing a converse to Observation~\ref{obs:cycle-periodicity}, characterizing periodic tensions as precisely those vectors that satisfy the $\boldsymbol{T}$-cycle periodicity property.

\begin{theorem}
    \label{thm:gcd-conjecture}
    A vector $x \in \mathbb R^A$ satisfies the $\boldsymbol{T}$-cycle periodicity property \eqref{eq:mpesp-cycle-periodicity} if and only if there is a vector $\pi \in \mathbb R^V$ such that for all $a = (i, j) \in A$ holds $\pi_j - \pi_i \equiv x_{a} \bmod T_a$ \textnormal{(MP\ref{item:mp})}.
\end{theorem}

To prove Theorem~\ref{thm:gcd-conjecture}, we need a few preparatory results and recall first a generalized version of the Chinese remainder theorem:

\begin{theorem}[\citealp{ore_general_1952}]
\label{thm:crt}
Let $a_1, \dots, a_n \in \mathbb Z$ and $T_1, \dots, T_n \in \mathbb N$. Then there is a number $\pi \in \mathbb Z$ that satisfies $\pi \equiv a_i \bmod T_i$ simultaneously for all $i \in \{1, \dots, n\}$ if and only if $a_i \equiv a_j \bmod \gcd\{T_i, T_j\}$ for all $i \neq j$.
\end{theorem}

Ore's proof of Theorem~\ref{thm:crt} immediately generalizes to real numbers, and is constructive: A solution $\pi$ can be computed by essentially applying the extended Euclidean algorithm. We now use the theorem to prove Theorem~\ref{thm:gcd-conjecture} for tournament graphs. A directed graph is a \emph{tournament graph} if there is exactly one arc between any pair of vertices. 

\begin{lemma}
    \label{lem:tournament}
	Let $G = (V, A)$ be a tournament graph and $x \in \mathbb R^A$ such that for all cycles $C$ in $G$ holds $\gamma_C^\top x \equiv 0 \bmod T_C$.
	Then $x$ is a periodic tension.
\end{lemma}
\begin{proof}
	We proceed by induction on the number $n$ of nodes. The situation for $n = 1$ is trivial. For the induction step, fix an arbitrary node $k \in V$. Reversing all outgoing arcs $a \in \delta^+(k)$ and replacing $x_a$ by $-x_a$, we can assume without loss of generality that $\delta^+(k) = \emptyset$. The subgraph on $V \setminus \{k\}$ is a tournament graph again, so that by induction hypothesis, the restriction of $x$ is a periodic tension for some periodic timetable $\pi$. We now want to construct $\pi_k$ such that
    \begin{equation}
        \label{eq:tournament-pi}
        \pi_k \equiv \pi_i + x_{ik} \mod \gcd\{T_i, T_k\} \quad  \text{ for all } i \in \delta^-(k).
    \end{equation}
	That is, we have to solve a simultaneous congruence, which by Theorem~\ref{thm:crt} has a solution if and only if
    \begin{equation}
        \label{eq:tournament-crt}
	    \pi_i + x_{ik} \equiv \pi_j + x_{jk} \mod \gcd\{T_i, T_j, T_k\} \quad \text{ for all } i, j \in \delta^-(k), i \neq j. 
    \end{equation}
	As $G$ is a tournament, the nodes $i$, $j$, $k$ constitute a cycle $C$, so that
    \begin{equation}
	x_{ik} \equiv x_{ij} + x_{jk}  \mod T_C = \gcd\{T_i, T_j, T_k\} \quad \text{ or }\quad
    x_{ik} \equiv -x_{ji} + x_{jk}  \mod T_C = \gcd\{T_i, T_j, T_k\},
    \end{equation}
	depending on the orientation of the arc between $i$ and $j$.
	Inserting that $\pi_j - \pi_i \equiv x_{ij} \bmod \gcd\{T_i, T_j\}$ or $\pi_j - \pi_i \equiv -x_{ji} \bmod \gcd\{T_i, T_j\}$ by induction hypothesis,  we find 
    \begin{equation}
	   x_{ik} \equiv \pi_j - \pi_i + x_{jk} \mod \gcd\{T_i, T_j, T_k\},
    \end{equation}
    implying \eqref{eq:tournament-crt} and hence the existence of a $\pi_k$ satisfying \eqref{eq:tournament-pi}.
\end{proof}

\begin{lemma}
    \label{lem:subgraph}
	Let $G = (V, A)$ be a graph and $a \in A$. If Theorem~\ref{thm:gcd-conjecture} holds for $G$, then it holds for the subgraph $G' = (V, A \setminus \{a\})$.
\end{lemma}
\begin{proof}
    By Observation~\ref{obs:cycle-periodicity}, it suffices to show the ``only if'' direction in Theorem~\ref{thm:gcd-conjecture}. Let $x' \in \mathbb R^{A \setminus \{a\}}$ such that $\gamma_{C'}^\top x \equiv 0 \bmod T_{C'}$
	for all cycles $C'$ in $G'$. If we can extend $x'$ to a vector $x \in \mathbb R^A$ such that $\gamma_C^\top x \equiv 0 \bmod T_C$ for all cycles $C$ with $a \in A(C)$, then $x$ is a periodic tension on $G$, and the corresponding periodic timetable will certify that $x'$ is a periodic tension on $G'$ as well. Let us therefore consider $a = (i, j)$ and look for an $x_a \in \mathbb R$ with
    \begin{equation}
        \label{eq:subgraph-path}
       x_a \equiv \gamma_P^\top x' \mod \gcd\{T_i \mid i \in P\} \quad \text{ for all } i\text{-}j\text{-paths } P \text{ in } G',
    \end{equation}
    where we encode a not necessarily directed path $P$ by an incidence vector $\gamma_P \in \{0, \pm 1\}^A$ in the same manner as for cycles.
    Since each cycle $C$ containing $a$ an be decomposed into an $i$-$j$-path $P$ and the arc $a$, and cycles not containing $a$ are taken care of by the hypothesis, \eqref{eq:subgraph-path} implies 
    \begin{equation}
      \gamma^\top_C x \equiv 0 \bmod \gcd\{T_i \mid i \in P\} = \gcd\{T_i \mid i \in C\} = T_C  \quad \text{ for all cycles } C \text{ in } G.
    \end{equation}
	By Theorem~\ref{thm:crt}, the simultaneous congruence \eqref{eq:subgraph-path} can be solved if and only if
    \begin{equation}
    \label{eq:subgraph-crt}
	\gamma_P^\top x' \equiv \gamma_Q^\top x' \mod \gcd\{T_i \mid i \in P \cup Q\} \quad  \text{ for all } i\text{-}j\text{-paths } P, Q \text{ in } G', P \neq Q.
    \end{equation}
	For any two $i$-$j$-paths $P$ and $Q$, the difference $\gamma_P - \gamma_Q = \gamma_{C'}$ is the incidence vector of some cycle $C'$ in $G'$, so that by assumption
	$$ \gamma_P^\top x' \equiv \gamma_Q^\top x' \mod T_{C'}.$$
	Since any node of $C'$ is a node of $p$ or of $q$, $\gcd\{T_i \mid i \in P \cup Q\}$ must divide $T_{C'}$, and we obtain \eqref{eq:subgraph-crt}.
\end{proof}

\begin{proof}[Proof of Theorem~\ref{thm:gcd-conjecture}]
    It is straightforward to deal with loops, parallel or antiparallel arcs, so that we can assume without loss of generality that the event-activity network $G$ is a subgraph of a tournament graph. We then invoke Lemma~\ref{lem:tournament} and Lemma~\ref{lem:subgraph}.
\end{proof}

The proof of Lemma~\ref{lem:tournament} also yields an incremental procedure for constructing a periodic timetable $\pi$ from a tension $x$, by iteratively adding nodes and solving systems of simultaneous congruences. We first describe such a procedure for \emph{chordal} graphs, followed by a discussion of its extension to general graphs.

A graph is called \emph{chordal} if its nodes admit a perfect elimination ordering $(v_1, \dots, v_n)$, i.e., for all $i \in \{2, \dots, n\}$, $v_i$ and its neighbors $v_j$ with $j < i$ form a clique. Lemma~\ref{lem:tournament} can then be generalized to chordal graphs, using a perfect elimination ordering for the induction. 
We illustrate such a procedure in Figure~\ref{fig:traversals-crt} using the instance of Figure~\ref{fig:traversals}. The event-activity network is chordal, with $(v_1, v_2, v_3, v_4)$ being a perfect elimination ordering. We obtain a timetable $\pi$ as follows:
\begin{equation}
    \label{eq:transversal-crt}
    \begin{aligned}
        & & \pi_1 &\coloneqq 0 \\
        \pi_2 \equiv 0 + 1 \bmod 10  \quad &\rightsquigarrow&  \pi_2 &\coloneqq 1 \\
        \pi_3 \equiv 1 + 7 \bmod 10 \quad \text{and} \quad \pi_3 \equiv 0 - 2 \bmod 10  \quad &\rightsquigarrow&  \pi_3 &\coloneqq 8 \\
        \pi_4 \equiv 8 + 16 \bmod 10 \quad \text{and} \quad \pi_4 \equiv 0 - 6 \bmod 20 \quad &\rightsquigarrow & \pi_4 &\coloneqq 14
    \end{aligned}
\end{equation}

\begin{figure}[h]
\centering
    % First subfigure
    \begin{subfigure}[t]{0.8\textwidth}
        \begin{center}
    	\begin{tikzpicture}[scale=2]
    		\tikzstyle{v} = [draw, circle, minimum width=20];
    		\tikzstyle{a} = [-stealth, line width=1];
    		\tikzstyle{t} = [midway, above, sloped]
    		\node[v, label={180:$T_1 = 20$}] (1) at (0, 0) {$v_1$};
    		\node[v, label={180:$T_2 = 10$}] (2) at (0, 1) {$v_2$};
    		\node[v, label={0:$T_3 = 10$}] (3) at (1, 1) {$v_3$};
    		\node[v, label={0:$T_4 = 20$}] (4) at (1, 0) {$v_4$};
    		\draw[a] (1) -- node[t] {$1$} (2);
    		\draw[a] (2) -- node[t] {$7$} (3);
    		\draw[a] (3) -- node[t] {$2$} (1);
    		\draw[a] (3) -- node[t] {$16$} (4);
    		\draw[a] (4) -- node[t] {$6$} (1);
    	\end{tikzpicture}
        \end{center}
        \caption{Instance with periodic tension as in Figure~\ref{fig:traversalsa}.}
    \end{subfigure}
    % 2nd subfigure
    \begin{subfigure}[t]{0.24\textwidth}
        \begin{center}
    	\begin{tikzpicture}[scale=2]
    		\tikzstyle{v} = [draw, circle, minimum width=20];
    		\tikzstyle{a} = [-stealth, line width=1];
    		\tikzstyle{t} = [midway, above, sloped]
    		\node[v, red] (1) at (0, 0) {$0$};
    		\node[v] (2) at (0, 1) {$ $};
    		\node[v] (3) at (1, 1) {$ $};
    		\node[v] (4) at (1, 0) {$ $};
    		\draw[a] (1) -- node[t] {$1$} (2);
    		\draw[a] (2) -- node[t] {$7$} (3);
    		\draw[a] (3) -- node[t] {$2$} (1);
    		\draw[a] (3) -- node[t] {$16$} (4);
    		\draw[a] (4) -- node[t] {$6$} (1);
    	\end{tikzpicture}
        \end{center}
        \caption{Starting at $v_1$.}
    \end{subfigure}
    % 3rd subfigure
    \begin{subfigure}[t]{0.24\textwidth}
        \begin{center}
    	\begin{tikzpicture}[scale=2]
    		\tikzstyle{v} = [draw, circle, minimum width=20];
    		\tikzstyle{a} = [-stealth, line width=1];
    		\tikzstyle{t} = [midway, above, sloped]
    		\node[v, blue] (1) at (0, 0) {$0$};
    		\node[v, red] (2) at (0, 1) {$1$};
    		\node[v] (3) at (1, 1) {$ $};
    		\node[v] (4) at (1, 0) {$ $};
    		\draw[a, red] (1) -- node[t] {$1$} (2);
    		\draw[a] (2) -- node[t] {$7$} (3);
    		\draw[a] (3) -- node[t] {$2$} (1);
    		\draw[a] (3) -- node[t] {$16$} (4);
    		\draw[a] (4) -- node[t] {$6$} (1);
    	\end{tikzpicture}
        \end{center}
        \caption{Adding $v_2$.}
    \end{subfigure}
    % 4th subfigure
    \begin{subfigure}[t]{0.24\textwidth}
        \begin{center}
    	\begin{tikzpicture}[scale=2]
    		\tikzstyle{v} = [draw, circle, minimum width=20];
    		\tikzstyle{a} = [-stealth, line width=1];
    		\tikzstyle{t} = [midway, above, sloped]
    		\node[v, blue] (1) at (0, 0) {$0$};
    		\node[v, blue] (2) at (0, 1) {$1$};
    		\node[v, red] (3) at (1, 1) {$8$};
    		\node[v] (4) at (1, 0) {$ $};
    		\draw[a, blue] (1) -- node[t] {$1$} (2);
    		\draw[a, red] (2) -- node[t] {$7$} (3);
    		\draw[a, red] (3) -- node[t] {$2$} (1);
    		\draw[a] (3) -- node[t] {$16$} (4);
    		\draw[a] (4) -- node[t] {$6$} (1);
    	\end{tikzpicture}
        \end{center}
        \caption{Adding $v_3$.}
    \end{subfigure}
    % 5th subfigure
    \begin{subfigure}[t]{0.24\textwidth}
        \begin{center}
    	\begin{tikzpicture}[scale=2]
    		\tikzstyle{v} = [draw, circle, minimum width=20];
    		\tikzstyle{a} = [-stealth, line width=1];
    		\tikzstyle{t} = [midway, above, sloped]
    		\node[v, blue] (1) at (0, 0) {$0$};
    		\node[v, blue] (2) at (0, 1) {$1$};
    		\node[v, blue] (3) at (1, 1) {$8$};
    		\node[v, red] (4) at (1, 0) {$14$};
    		\draw[a, blue] (1) -- node[t] {$1$} (2);
    		\draw[a, blue] (2) -- node[t] {$7$} (3);
    		\draw[a, blue] (3) -- node[t] {$2$} (1);
    		\draw[a, red] (3) -- node[t] {$16$} (4);
    		\draw[a, red] (4) -- node[t] {$6$} (1);
    	\end{tikzpicture}
        \end{center}
        \caption{Adding $v_4$.}
    \end{subfigure}
\caption{Incremental procedure to construct a timetable from a tension \eqref{eq:transversal-crt}.}
\label{fig:traversals-crt}
\end{figure}
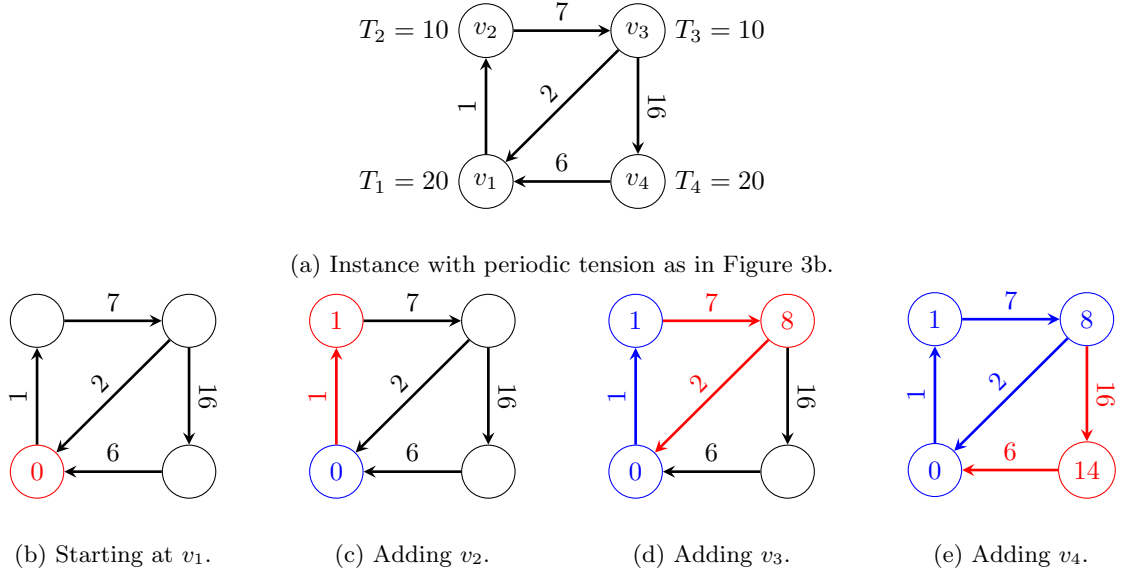

If the event-activity network $G$ is not chordal, then one can first resort to find a chordal completion, e.g., by adding arcs so that $G$ becomes a tournament graph. It is then however necessary to compute the missing tensions for the new arcs, which in the worst case requires to solve a system of an exponential number of simultaneous congruences \eqref{eq:subgraph-path}.

\section{MPESP for Public Transport Timetabling}
\label{sec:ptt}
In this section, we illustrate how MPESP can be used to model constraints that are typical in public transport timetabling. Subsequently, we discuss matters regarding passenger routing. Finally, we describe how to model local line synchronization.

\subsection{Transfer and Headway Constraints}
To illustrate the workings of activities within MPESP, consider an activity $(i,j)$ with $T_i=30$, $T_j=20$ and a tension $x_{ij}=5$. Figure~\ref{fig:slacksa} shows a possible timetable for this situation. In the context of public transport timetabling, $i$ and $j$ correspond to events of lines with frequencies of 2 per hour and 3 per hour, respectively. Figure~\ref{fig:slacksb} presents the same timetable, but in the typical PESP representation, where the global period is $T=\text{lcm}\{20,30\}=60.$

A \textbf{transfer constraint} between $i$ and $j$ with bounds $\tau^-$ and $\tau^+$ can be modeled as an activity $(i,j)$ with bounds $l_{ij}=\tau^-$ and $u_{ij}=\tau^+$. If there is only a lower bound on the transfer time, and total travel time is minimized in the objective, the bounds become $l_{ij}=\tau^-$ and $u_{ij}=\tau^-+T_{ij}-1$. Including this activity for $i$ and $j$ implies that there always exists some $(i_m,j_n)$ in the roll-out of the timetable where $x_{i_m,j_n}=x_{ij}$. This is also visible in Figure~\ref{fig:slacksb}: out of the six total transfer possibilities between the lines corresponding to $i$ and $j$, $(i_1,j_1)$ meets the desired transfer time. 
In case that $\frac{T}{T_i}$ and $\frac{T}{T_j}$ are not coprime, 
%In case $T_i$ divides $T_j$ or vice versa, 
there will be multiple transfers with the given transfer time. 

A \textbf{headway constraint} between $i$ and $j$ ensuring a minimum temporal separation of $h$ time units can be modeled as an activity $(i,j)$ with bounds $l_{ij}=h$ and $u_{ij}=T_{ij}-h$. Because $ x_{i_m,j_n}\equiv x_{ij} \bmod T_{ij}$ for all $m$ and $n$, this implies $h\leq x_{i_m,j_n} \leq T-h$, so minimum separation is actually ensured between \textit{all} pairs of events in the PESP representation. Again, this is also apparent from Figure~\ref{fig:slacksb}: assuming a minimum separation of $h\leq 5$ is required, all tensions are between 5 and 55 time units.

\begin{figure}[h]
    \centering
    % First row
    \begin{minipage}{0.45\textwidth}
        \centering
        \includegraphics[width=0.99\linewidth]{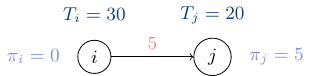}
        \subcaption{MPESP}
        \label{fig:slacksa}
    \end{minipage}
    \begin{minipage}{0.45\textwidth}
        \centering
        \includegraphics[width=0.99\linewidth]{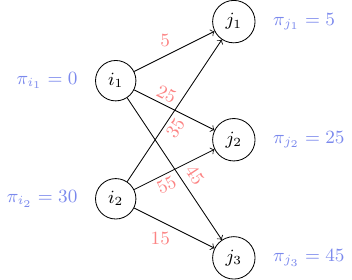}
        \subcaption{PESP}
        \label{fig:slacksb}
    \end{minipage}

    \caption{A timetable and tension represented using MPESP and PESP.}
    \label{fig:slacks}
\end{figure}

\subsection{Passenger Routing}
The objective of (M)PESP in public transport timetabling is typically to minimize passenger travel time. To do so, the weights $w$ are determined by computing a \textit{routing} of all passengers from their origins to destinations over the driving, dwelling and transferring activities with respect to the lower bounds $l$. After a solution $(\pi,x)$ has been found, the true travel time can be computed by routing all passengers according to the tension $x$. 

If the PESP representation is used to compute the routing and optimize the timetable, the objective value given by the PESP formulation always provides an upper bound on the true travel time with respect to the optimized tension $x$. However, this is not true for MPESP, as this approach can underestimate the travel times of passengers with two or more transfers.

This is illustrated in Figure~\ref{fig:nm}. A passenger traveling from station $A$ to station $D$ can either use a direct connection with the orange line, or transfer to the blue line at station $B$, and then transfer back to the orange line at station $C$. Figures~\ref{fig:nm2} and \ref{fig:nm3} depict the same timetable, represented using MPESP and PESP, respectively. According to the MPESP representation, the passenger should transfer twice, resulting in a total travel time of 25. However, the PESP representation shows that it is not possible to have two 5-minute transfers in the same path, such that the actual travel time is 55.

\begin{figure}[h]
    \centering

    % First row
    \begin{minipage}{0.8\textwidth}
        \centering
        \includegraphics[width=0.3\linewidth]{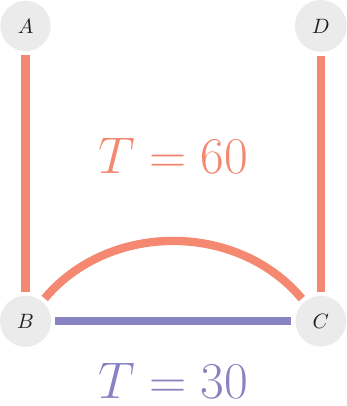}
        \subcaption{Two lines with different frequencies.}
        \label{fig:nm1}
    \end{minipage}

    \vspace{1em} % Add vertical space between rows (adjust as needed)

    % Second row
    \begin{minipage}{0.49\textwidth}
        \centering
        \includegraphics[width=\linewidth]{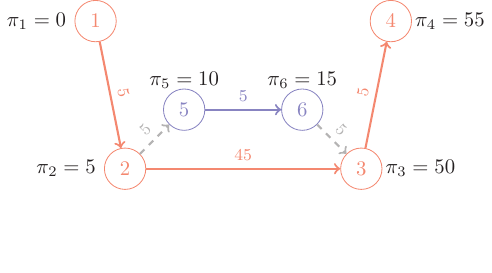}
        \subcaption{MPESP.}
        \label{fig:nm2}
    \end{minipage}
    \hfill
    \begin{minipage}{0.49\textwidth}
        \centering
        \includegraphics[width=\linewidth]{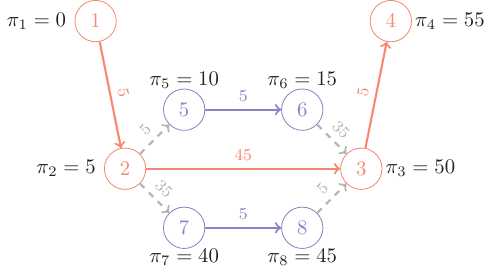}
        \subcaption{PESP.}
        \label{fig:nm3}
    \end{minipage}

    \caption{Example network and timetable where MPESP underestimates travel time for a passenger traveling from station $A$ to station $D$. Transfer arcs are depicted as dashed arcs.}
    \label{fig:nm}
\end{figure}

In real-life networks, the number of passengers with more than one transfer is typically small, such that MPESP is still a very good approximation of travel time while bringing the advantage of a more compact mathematical model. Finally, note that similar to PESP, it is possible to continue to iterate between routing and timetabling with MPESP, until convergence (see, e.g., \citealp{siebert2013experimental}).

\subsection{Local Line Synchronization}

A common feature in public transport networks is that multiple lines share a common segment—often in a city center or densely populated region—before branching off in different directions.  Even when these lines operate at the same frequency, the MPESP representation can still be used to obtain a more compact network representation in these situations. 

Consider Figure~\ref{fig:local}. All lines have a period of $T=60$ minutes, but between stations $A$ and $B$, where two lines are operated, services should depart and arrive exactly every 30 minutes. In the standard PESP, the departure and arrival events are all kept separate, and synchronization activities are added at $B$ for the arrivals towards $A$, and at $A$ for both arrivals and departures.

In contrast, the MPESP representation allows for merging events: the departures at $B$ towards $A$ can be combined, as can both arrivals and departures at $A$. This not only simplifies the event-activity network but also reduces the number of transfer arcs at $B$. Practically, this also makes sense, since passengers traveling from station $E$ or $F$ towards $A$ are (a priori) indifferent between the two lines that run from $B$ to $A$. 

However, to avoid simultaneous arrivals and departures of both lines on the shared segment, a soft synchronization activity must be introduced at station $B$ between the non-synchronized arrival and departure events. The bounds of these soft synchronization activity depend on the span of the dwell activities at $B$: e.g., if the dwell activities have bounds $[2,5]$ (so the span equals 3), the synchronization activities should have bounds $[27,33]$. More generally, if the dwell activities have span $s$, the synchronization activities should have bounds $[30-s,30+s]$.

\begin{figure}[!h]
    \centering
    % First subfigure
    \begin{subfigure}[t]{0.28\textwidth}
        \centering
        \includegraphics[width=0.8\linewidth]{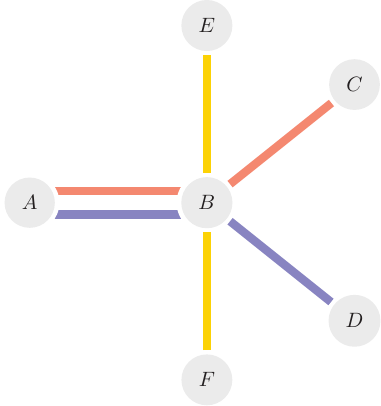}
        %\vspace{20pt}
        \caption{Three lines all running once per hour. Between stations $A$ and $B$, services should depart exactly every 30 minutes.}
        \label{fig:local0}
    \end{subfigure}
    \hfill
    % Second subfigure
    \begin{subfigure}[t]{0.35\textwidth}
        \centering
        \includegraphics[width=\linewidth]{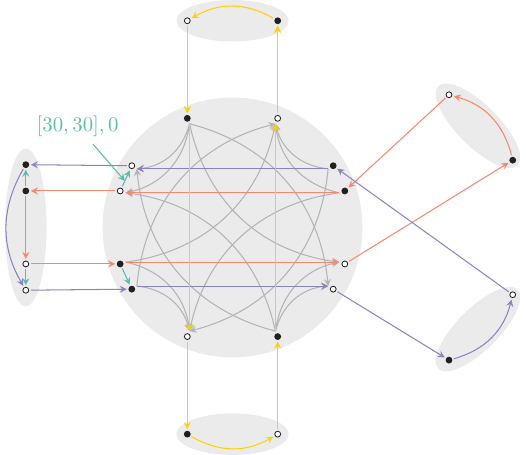}
        \caption{Event-activity network PESP.}
        \label{fig:local1}
    \end{subfigure}
     \hfill
    % Third subfigure
    \begin{subfigure}[t]{0.35\textwidth}
        \centering
        \includegraphics[width=\linewidth]{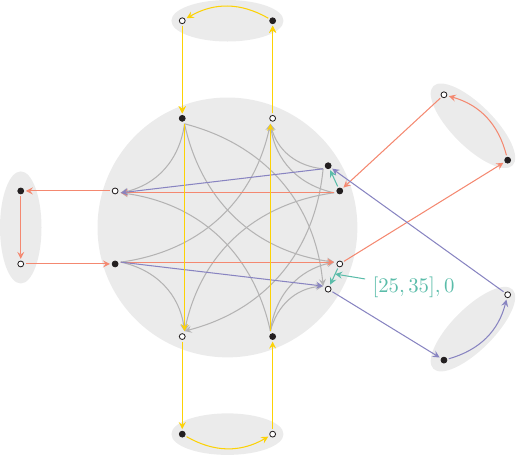}
        \caption{Event-activity network MPESP.}
        \label{fig:local2}
    \end{subfigure}
    \caption{Line network with local synchronization and the corresponding event-activity networks for PESP and MPESP.}
    \label{fig:local}
\end{figure}

\section{Computational Results}
\label{sec:cpu}

This section presents the results of a series of numerical experiments with the new formulation. First, we describe the used instances and computational set-up. Then, we compare the computational performance of the different formulations. Finally, we apply the formulation to solve periodic timetabling problems with integrated routing.

\subsection{Instances and Set-Up}
The instances are derived from the benchmarking library TimPassLib \citep{Schiewe2023TimPassLib}. While this library provides instances for integrated periodic timetabling and passenger routing, they can be transformed into (M)PESP instances by computing a fixed passenger routing with respect to the lower bounds of the activities. Excluding cases where all lines operate at a common frequency, this yields two artificial and five real-world instances, as shown in Table~\ref{tbl:instances}. 

\begin{table}[ht]
\centering
\caption{Instances used in the computational study.}
\label{tbl:instances}
		\begin{tabular}{crrcr}
			\hline
			Instance     & \multicolumn{1}{c}{Stations} & \multicolumn{1}{c}{Lines} & Period set $\mathcal{T}$                           & \multicolumn{1}{c}{lcm($\mathcal{T}$)} \\ \hline
			Toy          & 8                            & 6                         & \{15, 20, 30, 60\}                      & 60                                     \\
			Grid         & 25                           & 8                         & \{20, 30, 60\}                          & 60                                     \\
		 Saxony & 34                           & 8                         & \{30, 60\}                               & 60                                     \\
			Athens       & 51                           & 4                         & \{60, 75, 100\}                         & 300                                    \\
			Erding       & 51                           & 21                        & \{10, 15, 20, 30, 60\}                  & 60                                     \\
			Switzerland  & 140                          & 80                        & \{30, 60, 120\}                           & 120                                    \\
			Stuttgart    & 560                          & 156                       & \{300, 360, 450, 600, 900, 1200, 1800\} & 3600                                   \\ \hline
		\end{tabular}
	\end{table}

Saxony and Switzerland are harmonic instances, since their period sets are $\{30,60\}$ and $\{30,60,120\}$, respectively. According to the principles discussed in Section~\ref{sec:nice}, the remaining instances are converted into rooted instances to be solvable using the cycle formulation. To assess the benefits of MPESP, every MPESP instance is also transformed into a PESP instance, by using the least common multiple of all periods as the global period, and duplicating the nodes and arcs that occur at higher frequencies. To maintain equivalence between MPESP and PESP, the weights of transfer arcs are evenly distributed across their corresponding duplicated arcs.

Table~\ref{tbl:sizes} reports the size of the event-activity networks in all three representations: PESP, MPESP and rooted MPESP. The difference in size between the PESP and MPESP models is substantial, especially for larger instances like Stuttgart, where the arc count is reduced by over 90\%. The additional overhead to root the MPESP instances is minor—only two of the instances require an additional node for the least common multiple, and the increase in arcs is typically below 10. 

    \begin{table}[hb]
    \centering
    \caption{Size of the event-activity networks. NA = not applicable (harmonic).}
    \label{tbl:sizes}
		\begin{tabular}{ccrrrrrrrr}
			\hline
			&  & \multicolumn{2}{c}{PESP}                             & \multicolumn{1}{c}{} & \multicolumn{2}{c}{MPESP}                            & \multicolumn{1}{c}{} & \multicolumn{2}{c}{Rooted MPESP}                       \\ \cline{3-4} \cline{6-7} \cline{9-10} 
			Instance     &  & \multicolumn{1}{c}{Nodes} & \multicolumn{1}{c}{Arcs} & \multicolumn{1}{c}{} & \multicolumn{1}{c}{Nodes} & \multicolumn{1}{c}{Arcs} & \multicolumn{1}{c}{} & \multicolumn{1}{c}{Nodes} & \multicolumn{1}{c}{Arcs} \\ \hline
			Toy          &  & 156                       & 295                      &                      & 64                        & 62                       &                      & 64                        & 69                       \\
			Grid         &  & 392                       & 911                      &                      & 216                       & 301                      &                      & 216                       & 305                      \\
			Saxony &  & 412                       & 792                      &                      & 236                       & 302                      &                      & NA                       & NA                      \\
			Athens       &  & 964                       & 2527                     &                      & 256                       & 316                      &                      & 257                       & 322                      \\
			Erding       &  & 1132                      & 2618                     &                      & 492                       & 599                      &                      & 492                       & 602                      \\
			Switzerland  &  & 2234                      & 7724                     &                      & 1248                      & 2492                     &                      & NA                      & NA                     \\
			Stuttgart    &  & 21328                     & 127384                   &                      & 4696                      & 8295                     &                      & 4697                      & 8330                     \\ \hline
		\end{tabular}
	\end{table}

To generate a broader set of instances with varying complexity, 10 instances are created for each base instance by retaining only the top 10\%, 20\%, ..., up to 100\% of transfer arcs. The arcs are first sorted by weight, so for example, instance \texttt{grid-0.4} contains the 40\% of transfer arcs with the highest weights.

To comprehensively evaluate the effectiveness of the proposed approach, both the arc and cycle formulations are tested using the PESP and (rooted) MPESP representations, resulting in four distinct formulations. Cycle bases are constructed using the extension of the spanning tree heuristic described in Section~\ref{sec:nice}.  All experiments are performed with CPLEX 22.1, utilizing 32 cores of an AMD Rome 7H12 processor, with a maximum computation time of 1 hour per run. The complete source code and all test instances are publicly available at: \href{https://github.com/rn-van-lieshout/MPESP}{\texttt{github.com/rn-van-lieshout/MPESP}}.

\subsection{Comparing the Formulations}
\label{ss:compare}
In this section, we compare the numerical performance of the four formulations. For the small and medium-sized instances, Table~\ref{tbl:smallAndMediumAve} presents the number of optimal solutions, average computation time and average optimality gap of the different formulations. The detailed results of the individual runs are reported in Table~\ref{tbl:detailedSmall} in Appendix~\ref{app:res}. 

\begin{table}[ht]
\caption{Number of optimal solutions, average computation time and average optimality gap of the small and medium-sized instances.}
\label{tbl:smallAndMediumAve}
    \resizebox{\textwidth}{!}{%
\begin{tabular}{crrrrrrrrrrrrrrr}
\hline
             & \multicolumn{7}{c}{PESP}                                                                                                                                                                           & \multicolumn{1}{c}{} & \multicolumn{7}{c}{MPESP}                                                                                                                                                                          \\ \cline{2-8} \cline{10-16} 
             & \multicolumn{3}{c}{Arc-Based}                                                        & \multicolumn{1}{c}{} & \multicolumn{3}{c}{Cycle-Based}                                                      & \multicolumn{1}{c}{} & \multicolumn{3}{c}{Arc-Based}                                                        & \multicolumn{1}{c}{} & \multicolumn{3}{c}{Cycle-Based}                                                      \\ \cline{2-4} \cline{6-8} \cline{10-12} \cline{14-16} 
             Instance & \multicolumn{1}{c}{Opt.} & \multicolumn{1}{c}{T. (s)} & \multicolumn{1}{c}{Gap (\%)} & \multicolumn{1}{c}{} & \multicolumn{1}{c}{Opt.} & \multicolumn{1}{c}{T. (s)} & \multicolumn{1}{c}{Gap (\%)} & \multicolumn{1}{c}{} & \multicolumn{1}{c}{Opt.} & \multicolumn{1}{c}{T. (s)} & \multicolumn{1}{c}{Gap (\%)} & \multicolumn{1}{c}{} & \multicolumn{1}{c}{Opt.} & \multicolumn{1}{c}{T. (s)} & \multicolumn{1}{c}{Gap (\%)} \\ \hline
Toy          & 10/10                       & 245                          & 0.0                   &                      & 10/10                       & 1                            & 0.0                   &                      & 10/10                       & 2                            & 0.0                   &                      & 10/10                       & 0                            & 0.0                   \\
Grid         & 1/10                        & 2597                         & 4.1                   &                      & 4/10                        & 2162                         & 1.3                   &                      & 7/10                        & 1273                         & 0.1                   &                      & 10/10                       & 10                           & 0.0                   \\
Saxony & 5/10                        & 2150                         & 1.4                   &                      & 10/10                       & 3                            & 0.0                   &                      & 10/10                       & 56                           & 0.0                   &                      & 10/10                       & 1                            & 0.0                   \\
Athens       & 0/10                        & 3102                         & 8.3                   &                      & 2/10                        & 2916                         & 3.4                   &                      & 10/10                       & 43                           & 0.0                   &                      & 10/10                       & 5                            & 0.0                   \\
Erding       & 0/10                        & 3605                         & 2.2                   &                      & 1/10                        & 3268                         & 0.9                   &                      & 4/10                        & 2285                         & 0.1                   &                      & 10/10                       & 61                           & 0.0                   \\ \hline
\end{tabular}
}
\end{table}

Comparing the four formulations, it is immediately clear that the cycle formulation for MPESP outperforms all other formulations. This is the only formulation that solves all fifty small and medium-sized instances to optimality within the 1 hour time limit, and even does so in an average computation time of only 15 seconds. The detailed results in the appendix reveal that the other formulations are relatively close to the best formulation on the primal side, but terminate with large optimality gaps due to suffering from weak lower bounds. 

In agreement with the existing literature, the cycle  formulation consistently dominates the arc formulation in its computational performance. However, the choice of instance representation is more influential for the performance than the choice between arc or cycle formulation -- the second best formulation is the arc formulation on the MPESP representation, which only fails to prove optimality on nine of the instances, and has an average optimality gap of only 0.1\%.

Moving to larger instances, Table~\ref{tbl:large1} presents the objective and optimality gap obtained with the four formulations for the instances based on Switzerland and Stuttgart. None of these are solved within the 1 hour time limit, so computation times are not reported. Surprisingly, the superiority of the cycle formulation on the MPESP representation on the small and medium-sized instances does not carry over to these large instances. While this formulation still terminates with the smallest gap on some instances, it fails to even find feasible solutions on others. The arc formulation on MPESP has the most stable performance, but one some instances the cycle formulation on PESP finds a better primal solution.

\begin{table}[ht]
\centering
\caption{Objective values and optimality gaps obtained using the four formulations on the large instances. Only the first five digits (after rounding) of the objective values are shown.}
\label{tbl:large1}
\begin{tabular}{crrrrrrrrccr}
\hline
                         & \multicolumn{5}{c}{PESP}                                                                                                       & \multicolumn{1}{c}{} & \multicolumn{5}{c}{MPESP}                                                                                                     \\ \cline{2-6} \cline{8-12} 
                         & \multicolumn{2}{c}{Arc-Based}                      & \multicolumn{1}{c}{} & \multicolumn{2}{c}{Cycle-Based}                    & \multicolumn{1}{c}{} & \multicolumn{2}{c}{Arc-Based}                      &                      & \multicolumn{2}{c}{Cycle-Based}                     \\ \cline{2-3} \cline{5-6} \cline{8-9} \cline{11-12} 
Instance                 & \multicolumn{1}{c}{Obj.} & \multicolumn{1}{c}{ Gap (\%)} & \multicolumn{1}{c}{} & \multicolumn{1}{c}{Obj.} & \multicolumn{1}{c}{ Gap (\%)} & \multicolumn{1}{c}{} & \multicolumn{1}{c}{Obj.} & \multicolumn{1}{c}{ Gap (\%)} &                      & Obj.                      & \multicolumn{1}{c}{ Gap (\%)} \\ \hline
Switzerland-0.1 & 59787                    & 1.7                   &                      & 59746                    & 1.1                   &                      & 59721                    & 0.4                   & \multicolumn{1}{r}{} & \multicolumn{1}{r}{59687} & 0.1                   \\
Switzerland-0.2 & 62110                    & 4.8                   &                      & 61344                    & 3.0                   &                      & 61590                    & 2.6                   & \multicolumn{1}{r}{} & \multicolumn{1}{r}{61710} & 1.6                   \\
Switzerland-0.3 & 63349                    & 6.2                   &                      & 63479                    & 5.6                   &                      & 62602                    & 3.6                   & \multicolumn{1}{r}{} & \multicolumn{1}{r}{63568} & 3.4                   \\
Switzerland-0.4 & 64018                    & 7.1                   &                      & 63918                    & 5.7                   &                      & 64045                    & 4.9                   & \multicolumn{1}{r}{} & \multicolumn{1}{r}{64163} & 3.8                   \\
Switzerland-0.5 & 64921                    & 8.2                   &                      & 65298                    & 7.4                   &                      & 65070                    & 6.0                   & \multicolumn{1}{r}{} & \multicolumn{1}{r}{64882} & 4.4                   \\
Switzerland-0.6 & 65293                    & 8.6                   &                      & 65179                    & 7.3                   &                      & 65831                    & 6.8                   & \multicolumn{1}{r}{} & \multicolumn{1}{r}{66844} & 7.0                   \\
Switzerland-0.7 & 65624                    & 9.0                   &                      & 66046                    & 8.2                   &                      & 65186                    & 6.0                   & \multicolumn{1}{r}{} & \multicolumn{1}{r}{66765} & 6.9                   \\
Switzerland-0.8 & 65319                    & 8.5                   &                      & 66070                    & 8.0                   &                      & 65306                    & 6.2                   & \multicolumn{1}{r}{} & \multicolumn{1}{r}{66822} & 7.1                   \\
Switzerland-0.9 & 65779                    & 9.2                   &                      & 66802                    & 9.2                   &                      & 65444                    & 6.1                   & \multicolumn{1}{r}{} & -                         & \multicolumn{1}{c}{-}   \\
Switzerland-1.0 & 65773                    & 9.2                   &                      & 66428                    & 8.6                   &                      & 66970                    & 8.2                   & \multicolumn{1}{r}{} & -                         & \multicolumn{1}{c}{-}   \\ \hline
Stuttgart-0.1            & 44103                    & 41.0                  &                      & 43130                    & 18.6                  &                      & 42830                    & 1.5                   &                      & \multicolumn{1}{r}{42830} & 0.3                   \\
Stuttgart-0.2            & 49349                    & 59.1                  &                      & 45845                    & 33.5                  &                      & 45348                    & 3.0                   &                      & \multicolumn{1}{r}{45903} & 3.1                   \\
Stuttgart-0.3            & 49802                    & 71.2                  &                      & 47499                    & 38.4                  &                      & 46897                    & 4.5                   &                      & \multicolumn{1}{r}{47569} & 4.7                   \\
Stuttgart-0.4            & 51978                    & 79.2                  &                      & 48493                    & 43.1                  &                      & 48183                    & 5.6                   &                      & -                         & \multicolumn{1}{c}{-}   \\
Stuttgart-0.5            & 53014                    & 85.8                  &                      & 49802                    & 48.6                  &                      & 48901                    & 6.4                   &                      & -                         & \multicolumn{1}{c}{-}   \\
Stuttgart-0.6            & 53537                    & 90.1                  &                      & 50004                    & 52.9                  &                      & 48630                    & 5.7                   &                      & -                         & \multicolumn{1}{c}{-}   \\
Stuttgart-0.7            & 53272                    & 98.4                  &                      & 50315                    & 55.2                  &                      & 49537                    & 7.3                   &                      & -                         & \multicolumn{1}{c}{-}   \\
Stuttgart-0.8            & 53982                    & 99.3                  &                      & 50588                    & 56.9                  &                      & 48926                    & 6.1                   &                      & -                         & \multicolumn{1}{c}{-}   \\
Stuttgart-0.9            & 54417                    & 109.7                 &                      & 50627                    & 59.9                  &                      & 49407                    & 7.0                   &                      & -                         & \multicolumn{1}{c}{-}   \\
Stuttgart-1.0            & 54399                    & 119.0                 &                      & \multicolumn{1}{c}{-}    & \multicolumn{1}{c}{-}   &                      & 49591                    & 7.4                   &                      & -                         & \multicolumn{1}{c}{-}   \\ \hline
\end{tabular}
\end{table}

Since all transfer arcs in these instances are nonrestrictive (i.e., $u_a-l_a=T_a-1$), the space of feasible timetables is the same for all instances derived from the same network. Therefore, in the next experiment, the solution for instance \texttt{x-0.i} is used as a warm start for instance \texttt{x-0.(i+1)} for \texttt{i}$\,\leq 9$ and the solution for instance \texttt{x-0.9} is used as a warm start for instance \texttt{x-1.0}, always guaranteeing that the formulations find feasible solutions. Table~\ref{tbl:large2} reports the results of the formulations on the large instances with warm start. 

It can be concluded from Table~\ref{tbl:large2} that with warm start, the cycle formulation on the MPESP representation results in, on average, the smallest optimality gaps, again highlighting the strength of this formulation in finding good lower bounds. On the other hand, the the cycle formulation only finds the best primal solutions on the instances with 10\% and 20\% of all transfer arcs. For the other instances, the arc formulation for MPESP results in the best objective values. 

\begin{table}[ht]
\centering
\caption{Objective values and optimality gaps obtained using the four formulations using warm start on the large instances. Only the first five digits (after rounding) of the objective values are shown.}
\label{tbl:large2}
\begin{tabular}{crrrrrrrrrrr}
\hline
                         & \multicolumn{5}{c}{PESP}                                                                                                       & \multicolumn{1}{c}{} & \multicolumn{5}{c}{MPESP}                                                                                                    \\ \cline{2-6} \cline{8-12} 
                         & \multicolumn{2}{c}{Arc-Based}                      & \multicolumn{1}{c}{} & \multicolumn{2}{c}{Cycle-Based}                    & \multicolumn{1}{c}{} & \multicolumn{2}{c}{Arc-Based}                      & \multicolumn{1}{c}{} & \multicolumn{2}{c}{Cycle-Based}                    \\ \cline{2-3} \cline{5-6} \cline{8-9} \cline{11-12} 
Instance                 & \multicolumn{1}{c}{Obj.} & \multicolumn{1}{c}{Gap (\%)} & \multicolumn{1}{c}{} & \multicolumn{1}{c}{Obj.} & \multicolumn{1}{c}{Gap (\%)} & \multicolumn{1}{c}{} & \multicolumn{1}{c}{Obj.} & \multicolumn{1}{c}{Gap (\%)} & \multicolumn{1}{c}{} & \multicolumn{1}{c}{Obj.} & \multicolumn{1}{c}{Gap (\%)} \\ \hline
Switzerland-0.1 & 59865                    & 1.6                   &                      & 59739                    & 1.1                   &                      & 59692                    & 0.4                   &                      & 59688                    & 0.2                   \\
Switzerland-0.2 & 61728                    & 4.1                   &                      & 61539                    & 3.2                   &                      & 61557                    & 2.4                   &                      & 61546                    & 1.4                   \\
Switzerland-0.3 & 63287                    & 6.1                   &                      & 63217                    & 5.0                   &                      & 62869                    & 3.4                   &                      & 62933                    & 2.3                   \\
Switzerland-0.4 & 64193                    & 7.3                   &                      & 64220                    & 6.2                   &                      & 63563                    & 3.9                   &                      & 63765                    & 3.0                   \\
Switzerland-0.5 & 64676                    & 7.9                   &                      & 65335                    & 7.6                   &                      & 64180                    & 4.4                   &                      & 64378                    & 3.7                   \\
Switzerland-0.6 & 65858                    & 9.4                   &                      & 65403                    & 7.4                   &                      & 64586                    & 5.0                   &                      & 64666                    & 4.0                   \\
Switzerland-0.7 & 68202                    & 12.2                  &                      & 65607                    & 7.8                   &                      & 64821                    & 5.0                   &                      & 64946                    & 4.2                   \\
Switzerland-0.8 & 67300                    & 10.9                  &                      & 66097                    & 8.4                   &                      & 64981                    & 5.2                   &                      & 65106                    & 4.7                   \\
Switzerland-0.9 & 67128                    & 11.0                  &                      & 66335                    & 8.7                   &                      & 65061                    & 5.4                   &                      & 65186                    & 4.8                   \\
Switzerland-1.0 & 66145                    & 9.7                   &                      & 66085                    & 8.3                   &                      & 65080                    & 5.2                   &                      & 65201                    & 5.0                   \\
Average                  & 64838                    & 8.0                   &                      & 64358                    & 6.4                   &                      & 63639                    & 4.0                   &                      & 63741                    & 3.3                   \\ \hline
Stuttgart-0.1            & 44553                    & 41.7                  &                      & 43182                    & 18.8                  &                      & 42916                    & 1.8                   &                      & 42854                    & 0.4                   \\
Stuttgart-0.2            & 48469                    & 57.3                  &                      & 45856                    & 30.2                  &                      & 45495                    & 3.4                   &                      & 45468                    & 2.0                   \\
Stuttgart-0.3            & 51270                    & 64.9                  &                      & 47574                    & 39.1                  &                      & 46762                    & 4.1                   &                      & 46874                    & 3.2                   \\
Stuttgart-0.4            & 52037                    & 71.5                  &                      & 48553                    & 43.5                  &                      & 47576                    & 4.3                   &                      & 47725                    & 4.0                   \\
Stuttgart-0.5            & 53278                    & 79.4                  &                      & 49292                    & 46.6                  &                      & 48059                    & 4.6                   &                      & 48233                    & 4.3                   \\
Stuttgart-0.6            & 54264                    & 80.2                  &                      & 49815                    & 52.2                  &                      & 48350                    & 4.9                   &                      & 48474                    & 4.6                   \\
Stuttgart-0.7            & 54137                    & 84.7                  &                      & 50491                    & 55.0                  &                      & 48509                    & 5.0                   &                      & 48634                    & 5.0                   \\
Stuttgart-0.8            & 53634                    & 88.3                  &                      & 50588                    & 57.9                  &                      & 48579                    & 5.3                   &                      & 48709                    & 5.5                   \\
Stuttgart-0.9            & 54097                    & 89.3                  &                      & 50627                    & 60.8                  &                      & 48601                    & 5.3                   &                      & 48734                    & 5.7                   \\
Stuttgart-1.0            & 55184                    & 93.2                  &                      & 50633                    & 62.7                  &                      & 48604                    & 5.1                   &                      & 48737                    & 5.6                   \\
Average                  & 52092                    & 75.1                  &                      & 48661                    & 46.7                  &                      & 47345                    & 4.4                   &                      & 47444                    & 4.0                   \\ \hline
\end{tabular}
\end{table}

\subsection{Timetabling with Integrated Routing}
\label{sss:ttir}
The final experiment uses the MPESP representation to solve benchmark instances for the integrated periodic timetabling and passenger routing problem. In other words, instead of using fixed activity weights as in the previous experiments, the weight of an activity equals the number of passengers that uses that activity in their shortest path from origin to destination based on the tension $x$. The instances are solved using the iterative algorithm described in Section~\ref{sec:ptt}, with the refinement that only the top $k0\%$ transfer arcs are included in the timetabling step, where $k$ starts at 1 and is incremented in every iteration until $k$ equals 10. Since the focus is on finding good primal solutions, the large instances use the arc formulation, while the other instances use the cycle formulation. The Athens instance is omitted as it contains ``imperfect" synchronization constraints, modeling a service with a frequency of 4 in a 150-minute period with activities where $l=u=37$ while $150/4=37.5$. 

Table~\ref{tbl:ttir} reports the average passenger travel time and the gap to a lower bound obtained by computing the shortest path of all passengers with respect to $l$. For the Toy instance, the algorithm replicates the optimal solution found by \cite{SchieweSchoebel2020} using a specialized algorithm for timetabling with integrated routing. For four out of the five remaining instances, new best known solutions are found. These solutions have been independently verified and posted on the TimPassLib website. For Switzerland and Stuttgart, the new solutions improve the previous bests by 1.0\% and 1.4\%, respectively — substantial gains in the context of periodic timetabling.

\begin{table}[ht]
\centering
\caption{The average passenger travel time and gap to the lower bound for timetabling with integrated routing. The asterisk indicates that the found solution is a new best known solution.}
\label{tbl:ttir}
	\begin{tabular}{crr}
		\hline
		Instance     & \multicolumn{1}{c}{Travel Time (min.)} & \multicolumn{1}{c}{Gap (\%)} \\ \hline
		Toy          & 7.29\textcolor{white}{*}                                 & 0.0                   \\
		Grid         & 19.33*                                  & 2.8                   \\
		Saxony & 5.71\textcolor{white}{*}                                   & 3.0                   \\
		Erding       & 21.96*                                  & 0.4                   \\
		Switzerland  & 46.47*                                  & 4.1                   \\
		Stuttgart    & 20.38*                                  & 6.7                   \\ \hline
	\end{tabular}

\end{table}

\section{Conclusion}
\label{sec:con}
This paper developed a cycle formulation for periodic event scheduling problems with heterogeneous event frequencies. The new formulation operates on a more compact network compared to existing approaches, resulting in a stronger mathematical model.
On the theoretical side, we have characterized sharp trees, developed the notion of rooted instance and hence a cycle formulation for arbitrary instances, and proved a multiperiodic version of the cycle periodicity property.
Computational experiments on both artificial and real-life instances confirmed the superior performance of the new formulation. These results provide compelling evidence that implicitly modeling event frequencies—rather than duplicating events and introducing synchronization activities—is the more effective strategy.

% I conclude by presenting a conjecture inspired by the observed structural properties of tensions and timetables. Specifically, while this paper proves that the existence of a periodic tension guarantees a feasible periodic timetable when the instance contains a $\boldsymbol{T}$-nice spanning tree, I hypothesize that this implication holds more generally:
% \begin{conjecture*}
%     If $x$ satisfies $\gamma_C^\top x \equiv 0 \mod T_C$ for all cycles $C$ in $G$, there exists a feasible timetable $\pi$. 
% \end{conjecture*}

%First show necessity. 

%Would it be possible to extend a feasible tension on the ``nice" part of an instance to a feasible tension on the entire instance? You could construct a spanning forest which is maximally nice and then try to extend the tension. you miss some of the gcd>1 cycles that way...

%Idea for proving the gcd conjecture: try to extend the solution to the nice instance that you can create from it. For every arc, in the added part, you should check the periods of all fundamental cycles in which it appears. It seems that you can order the cycles that this generates. 

%First prove gcd conjecture for instance where all cycles have gcd 1. Implies that for any x we can find pi. 

\footnotesize
\vspace{10pt}
\textbf{Acknowledgements} Removed. 
%\noindent\textbf{Acknowledgements} We are grateful to Enrico Bortoletto for his valuable comments, which significantly improved this paper, and to Christian Liebchen for pointing out references. We also would like to thank Berenike Masing for providing the TikZ code used to generate Figures 1, 2, and 13.
\normalsize

\bibliography{refs}

\newpage
\appendix
\section{Detailed Computational Results}
\label{app:res}

\begin{table}[h]
\caption{Detailed results for the small and medium-sized instances. For Saxony, Athens and Erding instances, only the first five digits (after rounding) of the objective values are shown.}
\label{tbl:detailedSmall}
	    \resizebox{\textwidth}{!}{%

\begin{tabular}{crrrrrrrrrrrrrrr}
\hline
                      & \multicolumn{7}{c}{PESP}                                                                                                                                                                     & \multicolumn{1}{c}{} & \multicolumn{7}{c}{MPESP}                                                                                                                                                                  \\ \cline{2-8} \cline{10-16} 
                      & \multicolumn{3}{c}{Arc-Based}                                                     & \multicolumn{1}{c}{} & \multicolumn{3}{c}{Cycle-Based}                                                   & \multicolumn{1}{c}{} & \multicolumn{3}{c}{Arc-Based}                                                     & \multicolumn{1}{c}{} & \multicolumn{3}{c}{Cycle-Based}                                                   \\ \cline{2-4} \cline{6-8} \cline{10-12} \cline{14-16} 
Instance              & \multicolumn{1}{c}{Obj.} & \multicolumn{1}{c}{Time (s)} & \multicolumn{1}{c}{Gap (\%)} & \multicolumn{1}{c}{} & \multicolumn{1}{c}{Obj.} & \multicolumn{1}{c}{Time (s)} & \multicolumn{1}{c}{Gap (\%)} & \multicolumn{1}{c}{} & \multicolumn{1}{c}{Obj.} & \multicolumn{1}{c}{Time (s)} & \multicolumn{1}{c}{Gap (\%)} & \multicolumn{1}{c}{} & \multicolumn{1}{c}{Obj.} & \multicolumn{1}{c}{Time (s)} & \multicolumn{1}{c}{Gap (\%)} \\ \hline
Toy-0.1      & 14758                    & 2                            & 0.0                   &                      & 14758                    & 0                            & 0.0                   &                      & 14758                    & 1                            & 0.0                   &                      & 14758                    & 0                            & 0.0                   \\
Toy-0.2      & 15058                    & 1                            & 0.0                   &                      & 15058                    & 1                            & 0.0                   &                      & 15058                    & 1                            & 0.0                   &                      & 15058                    & 0                            & 0.0                   \\
Toy-0.3      & 15328                    & 1                            & 0.0                   &                      & 15328                    & 0                            & 0.0                   &                      & 15328                    & 2                            & 0.0                   &                      & 15328                    & 0                            & 0.0                   \\
Toy-0.4      & 15598                    & 2                            & 0.0                   &                      & 15598                    & 1                            & 0.0                   &                      & 15598                    & 1                            & 0.0                   &                      & 15598                    & 0                            & 0.0                   \\
Toy-0.5      & 15808                    & 4                            & 0.0                   &                      & 15808                    & 0                            & 0.0                   &                      & 15808                    & 3                            & 0.0                   &                      & 15808                    & 0                            & 0.0                   \\
Toy-0.6      & 16018                    & 24                           & 0.0                   &                      & 16018                    & 1                            & 0.0                   &                      & 16018                    & 0                            & 0.0                   &                      & 16018                    & 0                            & 0.0                   \\
Toy-0.7      & 16207                    & 146                          & 0.0                   &                      & 16207                    & 1                            & 0.0                   &                      & 16207                    & 3                            & 0.0                   &                      & 16207                    & 1                            & 0.0                   \\
Toy-0.8      & 16396                    & 149                          & 0.0                   &                      & 16396                    & 3                            & 0.0                   &                      & 16396                    & 2                            & 0.0                   &                      & 16396                    & 1                            & 0.0                   \\
Toy-0.9      & 16426                    & 1640                         & 0.0                   &                      & 16426                    & 1                            & 0.0                   &                      & 16426                    & 2                            & 0.0                   &                      & 16426                    & 1                            & 0.0                   \\
Toy-1.0      & 16456                    & 477                          & 0.0                   &                      & 16456                    & 2                            & 0.0                   &                      & 16456                    & 1                            & 0.0                   &                      & 16456                    & 0                            & 0.0                   \\ \hline
Grid-0.1        & 43797                    & 81                           & 0.0                   &                      & 43797                    & 1                            & 0.0                   &                      & 43797                    & 1                            & 0.0                   &                      & 43797                    & 0                            & 0.0                   \\
Grid-0.2        & 44389                    & 3601                         & 0.0                   &                      & 44389                    & 3                            & 0.0                   &                      & 44389                    & 2                            & 0.0                   &                      & 44389                    & 0                            & 0.0                   \\
Grid-0.3        & 44958                    & 3625                         & 1.48                  &                      & 44958                    & 50                           & 0.0                   &                      & 44958                    & 18                           & 0.0                   &                      & 44958                    & 1                            & 0.0                   \\
Grid-0.4        & 45358                    & 2411                         & 2.8                   &                      & 45358                    & 483                          & 0.0                   &                      & 45358                    & 40                           & 0.0                   &                      & 45358                    & 1                            & 0.0                   \\
Grid-0.5        & 45847                    & 3623                         & 3.8                   &                      & 45847                    & 3616                         & 0.4                   &                      & 45847                    & 48                           & 0.0                   &                      & 45847                    & 1                            & 0.0                   \\
Grid-0.6        & 46281                    & 2671                         & 5.0                   &                      & 46222                    & 3626                         & 1.0                   &                      & 46222                    & 276                          & 0.0                   &                      & 46222                    & 2                            & 0.0                   \\
Grid-0.7        & 46742                    & 3614                         & 5.8                   &                      & 46779                    & 3641                         & 2.2                   &                      & 46742                    & 1448                         & 0.0                   &                      & 46742                    & 11                           & 0.0                   \\
Grid-0.8        & 47447                    & 1995                         & 7.3                   &                      & 47093                    & 3636                         & 2.8                   &                      & 47047                    & 3634                         & 0.4                   &                      & 47047                    & 20                           & 0.0                   \\
Grid-0.9        & 47404                    & 2631                         & 6.7                   &                      & 47327                    & 3624                         & 3.2                   &                      & 47290                    & 3633                         & 0.5                   &                      & 47290                    & 29                           & 0.0                   \\
Grid-1.0        & 47846                    & 1721                         & 8.0                   &                      & 47505                    & 2941                         & 3.2                   &                      & 47426                    & 3627                         & 0.4                   &                      & 47426                    & 37                           & 0.0                   \\ \hline
Saxony-0.1    & 17498                    & 13                           & 0.0                   &                      & 17498                    & 0                            & 0.0                   &                      & 17498                    & 1                            & 0.0                   &                      & 17498                    & 0                            & 0.0                   \\
Saxony-0.2    & 17873                    & 5                            & 0.0                   &                      & 17873                    & 0                            & 0.0                   &                      & 17873                    & 1                            & 0.0                   &                      & 17873                    & 0                            & 0.0                   \\
Saxony-0.3    & 18695                    & 707                          & 0.0                   &                      & 18695                    & 1                            & 0.0                   &                      & 18695                    & 3                            & 0.0                   &                      & 18695                    & 0                            & 0.0                   \\
Saxony-0.4    & 19105                    & 1071                         & 0.0                   &                      & 19105                    & 1                            & 0.0                   &                      & 19105                    & 5                            & 0.0                   &                      & 19105                    & 0                            & 0.0                   \\
Saxony-0.5    & 19263                    & 3444                         & 0.0                   &                      & 19263                    & 1                            & 0.0                   &                      & 19263                    & 6                            & 0.0                   &                      & 19263                    & 1                            & 0.0                   \\
Saxony-0.6    & 19501                    & 3603                         & 0.7                   &                      & 19501                    & 2                            & 0.0                   &                      & 19501                    & 28                           & 0.0                   &                      & 19501                    & 0                            & 0.0                   \\
Saxony-0.7    & 19644                    & 3624                         & 2.4                   &                      & 19644                    & 4                            & 0.0                   &                      & 19644                    & 68                           & 0.0                   &                      & 19644                    & 1                            & 0.0                   \\
Saxony-0.8    & 19738                    & 3621                         & 2.4                   &                      & 19738                    & 4                            & 0.0                   &                      & 19738                    & 108                          & 0.0                   &                      & 19738                    & 1                            & 0.0                   \\
Saxony-0.9    & 19861                    & 3624                         & 4.2                   &                      & 19861                    & 8                            & 0.0                   &                      & 19861                    & 171                          & 0.0                   &                      & 19861                    & 1                            & 0.0                   \\
Saxony-1.0    & 19891                    & 1782                         & 4.5                   &                      & 19891                    & 8                            & 0.0                   &                      & 19891                    & 169                          & 0.0                   &                      & 19891                    & 2                            & 0.0                   \\ \hline
Athens-0.1       & 23037                    & 1863                         & 3.9                   &                      & 23037                    & 1                            & 0.0                   &                      & 23037                    & 1                            & 0.0                   &                      & 23037                    & 1                            & 0.0                     \\
Athens-0.2       & 23361                    & 2082                         & 5.3                   &                      & 23361                    & 19                           & 0.0                   &                      & 23361                    & 3                            & 0.0                   &                      & 23361                    & 0                            & 0.0                   \\
Athens-0.3       & 23589                    & 3608                         & 7.1                   &                      & 23589                    & 3632                         & 2.4                   &                      & 23589                    & 23                           & 0.0                   &                      & 23589                    & 0                            & 0.0                   \\
Athens-0.4       & 23757                    & 3609                         & 8.3                   &                      & 23757                    & 3630                         & 3.7                   &                      & 23757                    & 32                           & 0.0                   &                      & 23757                    & 1                            & 0.0                   \\
Athens-0.5       & 23847                    & 2679                         & 9.3                   &                      & 23847                    & 3669                         & 4.1                   &                      & 23847                    & 23                           & 0.0                   &                      & 23847                    & 1                            & 0.0                   \\
Athens-0.6       & 24292                    & 3060                         & 10.2                  &                      & 23904                    & 3643                         & 4.7                   &                      & 23904                    & 24                           & 0.0                   &                      & 23904                    & 1                            & 0.0                   \\
Athens-0.7       & 23948                    & 3608                         & 9.5                   &                      & 23937                    & 3653                         & 4.6                   &                      & 23937                    & 33                           & 0.0                   &                      & 23937                    & 5                            & 0.0                   \\
Athens-0.8       & 23968                    & 3295                         & 9.4                   &                      & 23956                    & 3637                         & 4.8                   &                      & 23956                    & 34                           & 0.0                   &                      & 23956                    & 1                            & 0.0                   \\
Athens-0.9       & 24016                    & 3607                         & 10.1                  &                      & 23967                    & 3653                         & 5.0                   &                      & 23967                    & 113                          & 0.0                   &                      & 23967                    & 23                           & 0.0                   \\
Athens-1.0       & 23972                    & 3604                         & 10.0                  &                      & 23972                    & 3625                         & 5.0                   &                      & 23971                    & 146                          & 0.0                   &                      & 23971                    & 14                           & 0.0                   \\ \hline
Erding-0.1 & 11891                    & 3612                         & 0.9                   &                      & 11891                    & 1                            & 0.0                   &                      & 11891                    & 1                            & 0.0                   &                      & 11891                    & 0                            & 0.0                   \\
Erding-0.2 & 11928                    & 3608                         & 1.6                   &                      & 11928                    & 3616                         & 0.1                   &                      & 11928                    & 3                            & 0.0                   &                      & 11929                    & 1                            & 0.0                   \\
Erding-0.3 & 11956                    & 3606                         & 1.9                   &                      & 11956                    & 3619                         & 0.5                   &                      & 11956                    & 296                          & 0.0                   &                      & 11956                    & 0                            & 0.0                   \\
Erding-0.4 & 11977                    & 3605                         & 2.1                   &                      & 11977                    & 3624                         & 0.7                   &                      & 11977                    & 751                          & 0.0                   &                      & 11977                    & 1                            & 0.0                   \\
Erding-0.5 & 11997                    & 3604                         & 2.5                   &                      & 11994                    & 3648                         & 1.1                   &                      & 11994                    & 3618                         & 0.1                   &                      & 11994                    & 1                            & 0.0                   \\
Erding-0.6 & 12012                    & 3605                         & 2.6                   &                      & 12011                    & 3640                         & 1.3                   &                      & 12011                    & 3638                         & 0.1                   &                      & 12011                    & 2                            & 0.0                   \\
Erding-0.7 & 12023                    & 3604                         & 2.6                   &                      & 12020                    & 3622                         & 1.2                   &                      & 12020                    & 3631                         & 0.1                   &                      & 12020                    & 8                            & 0.0                   \\
Erding-0.8 & 12024                    & 3603                         & 2.6                   &                      & 12024                    & 3643                         & 1.3                   &                      & 12026                    & 3636                         & 0.3                   &                      & 12023                    & 7                            & 0.0                   \\
Erding-0.9 & 12030                    & 3603                         & 2.8                   &                      & 12030                    & 3645                         & 1.4                   &                      & 12029                    & 3647                         & 0.3                   &                      & 12026                    & 28                           & 0.0                   \\
Erding-1.0 & 12030                    & 3603                         & 2.6                   &                      & 12033                    & 3620                         & 1.3                   &                      & 12029                    & 3630                         & 0.3                   &                      & 12027                    & 563                          & 0.0                   \\ \hline
\end{tabular}
}
\end{table}
\end{document}